\documentclass[11pt,twoside,a4paper,reqno]{amsart}
\usepackage[utf8]{inputenc}
\usepackage{amsmath,stackengine}
\usepackage{amssymb}
\usepackage{amsthm}
\usepackage[english]{babel}
\usepackage{tikz-cd}
\usepackage{url}
\usepackage{parskip}
\usepackage{enumerate}
\usepackage[normalem]{ulem}
\newtheorem{theorem}{Theorem}

\newtheorem{definition}[theorem]{Definition}

\newtheorem{lemma}[theorem]{Lemma}
\newtheorem{remark}[theorem]{Remark}
\newtheorem{proposition}[theorem]{Proposition}
\newtheorem{problem}{Problem}

\newcommand{\Var}{\text{Var}}
\newcommand{\as}[1]{\text{as } #1 \rightarrow\infty}
\renewcommand{\d}{\, \text{d}}
\newcommand{\s}[2]{(#1_#2)_{#2 \geq 1}}
\newcommand{\fceil}[2]{\left \lceil \frac{#1}{#2}  \right \rceil}
\newcommand{\ffloor}[2]{\left \lfloor \frac{#1}{#2}  \right \rfloor}
\newcommand{\F}{\mathcal{F}}
\newcommand{\R}{\mathbb{R}}
\newcommand{\Q}{\mathbb{Q}}

\newcommand{\Z}{\mathbb{Z}}
\newcommand{\T}{\mathbb{T}}
\newcommand{\N}{\mathbb{N}}
\renewcommand{\P}{\mathbb{P}}

\renewcommand{\to}{\rightarrow}
\usepackage{makecell}
\usepackage[margin=1in]{geometry}
\usepackage{todonotes}
\usepackage{soul}
\usepackage{tabu}

\title[Trimmed ergodic sums over rotations]{Trimmed ergodic sums for non-integrable functions with power singularities over irrational rotations}
\begin{document}

\title[Trimmed ergodic sums over rotations]{Trimmed ergodic sums for non-integrable functions with power singularities over irrational rotations}

\author{M. Auer}
\address{Department of Mathematics \\ University of Maryland College Park, USA}
\email{\href{mailto:mauer96@umd.edu}{mauer96@umd.edu}}

\author{T. I. Schindler}
\address{Faculty of Mathematics and Computer Science, Jagiellonian University \newline
\indent ul.~ {\L}ojasiewicza 6, 30-348 Krakow, Poland \newline 
\indent and \newline 
\indent Department of Mathematics and Statistics,
University of Exeter \newline
\indent Exeter EX4 4QF, United Kingdom
}
\email{\href{mailto:tanja.schindler@uj.edu.pl}{tanja.schindler@uj.edu.pl}}
\email{\href{mailto:t.schindler@exeter.ac.uk}{t.schindler@exeter.ac.uk}}

\thanks{The authors would like to thank Dmitry Dolgopyat and Adam
Kanigowski for helpful discussions and the Erwin Schr\"odinger
International Institute for Mathematical Physics in Vienna where this
work was initiated for their hospitality.
This project was supported by a grant from the priority research area SciMat under 
the Strategic Programme Excellence Initiative at Jagiellonian University (TS), and by the
European Union’s Horizon Europe
research and innovation programme under the Marie Sk{\l}odowska-Curie grant
agreement ErgodicHyperbolic - 101151185 (TS)}
\date{}
\maketitle
\begingroup
\leftskip4em
\rightskip\leftskip
\begin{small}
\paragraph*{Abstract}

Studying Birkhoff sums of non-integrable functions involves the challenge of large observations depending on the sampled orbit, which prevents pointwise limit theorems. To address this issue, the largest observations are removed, this process is commonly known as trimming. While this method is well studied for independent identically distributed sequences and systems with strong mixing behaviour, this paper focuses on irrational rotations of $\T$. In this setting we establish trimmed weak and strong laws for the functions $\frac{1}{x}$ and $\frac{1}{x^{\beta}}$ with $\beta>1$, providing explicit conditions on the rotation angle.

\end{small}
\par
\endgroup

\part{Results}

\section{Introduction}

\subsection{Motivation}

Let $(X, T,\mu)$ be a probability preserving ergodic dynamical system. The much celebrated Birkhoff Ergodic Theorem states that, given an integrable function $f\in L^1$, the average of observations over a long time approximates the integral, more formally
\begin{equation*}
    \lim_{N\rightarrow\infty} \frac{S_N(f)}{N} = \int f \d\mu  \;\;\; \text{ almost surely},
\end{equation*}
where $S_N(f)= \sum_{n=0}^{N-1} f \circ T^n$. In this paper, we will investigate the question 
\begin{equation*}
    \text{How do we describe asymptotics of } S_N(f) \text{ if } \int |f| \d\mu =\infty?
\end{equation*}

For simplicity, we shall restrict ourselves to functions that are non-negative and finite almost everywhere, denote the collection of these functions by
\begin{equation*}
\F=\{f:X\rightarrow [0,\infty] \text{ measurable} \;|\; f\not \in L^1(\mu), \text{ but } \mu(f=\infty)=0\}.
\end{equation*}
Often, in the situation when $X$ is a metric space, we will also require that $f\in \F$ is continuous\footnote{In this case we endow $[0,\infty]$ with the metric $d(s,t)=|e^{-s}-e^{-t}|$, for $s,t\in [0,\infty]$, making it a compact metric space. Note that $f$ is continuous if and only if, for every $s\in[0,\infty)$, the sets $\{f<s\}$ and $\{f>s\}$ are open.}.

The most naive way to approach this problem would be to normalize by a factor different from $N$, say $d_N>0$. However, Aaronson's "anti-convergence" result (\cite{Aaronson1977OnTE}) asserts that this is impossible: for every normalising sequence $\s{d}{N}$ it holds that
\begin{equation*}
     \limsup_{N\rightarrow\infty} \frac{S_N(f)}{d_N} = \infty \; \text{ or }\;  \liminf_{N\rightarrow\infty} \frac{S_N(f)}{d_N} =0 \;\;\;\text{ almost surely},
\end{equation*}
meaning that we either underestimate or overestimate the sum infinitely often. This non-convergence is caused by large observations, meaning $n$ such that $f(T^n x)$ is big, which may or may not occur depending on $x$. To circumvent this problem, a common method is to exclude the largest observations altogether. This method is referred to as trimming. 

More formally, for $N\geq 1$ and $0\leq k(N) \leq N$, define the \textit{trimmed sum} as
\begin{equation*}%\label{trimdef}
    S_N^{k(N)}(f)(x)=\sum_{n=0}^{N-1} f(T^n x) - \max_{0\leq n_1 < n_2 <...< n_{k(N)} \leq N-1} \sum_{j=1}^{k(N)} f ( T^{n_j} x).
\end{equation*}
We will also refer to $k(N)$ as the \textit{trimming sequence}. For simplicity set $S_N^M(f)=0$ if $M>N$.

The central problem we study is the following;
\begin{problem}\label{p1} 
Do there exist positive $k(N)=o(N)$ and $d_N>0$ such that 
\begin{equation}\label{strongtrimdef}
\lim_{N\rightarrow \infty} \frac{S_N^{k(N)}(f)(x)}{d_N} = 1 \;\;\; \text{for $\mu$-a.e. } x\in X?
\end{equation}
\end{problem}

If \eqref{strongtrimdef} holds, we refer to it as a trimmed strong law.

\begin{remark}\label{p1rem}
Note that $S_N^N(f)=0$ by definition. Furthermore, heuristically speaking, if $c\in (0,1)$, $f$ is continuous and $k(N)\sim c N$ (i.e.\ $\lim_{N\to\infty} k(N)/(cN)=1$) then
\begin{equation*}
S^{k(N)}_N (f) \sim S_N( \min(K,f))\text{ almost surely,}
\end{equation*}
where $K>0$ is such that $\mu(f>K)=c$.

This explains why we choose $k(N) = o(N)$ in Problem \ref{p1}; if $k(N) \sim cN$ then we lose information on $f$, making $S_N^{k(N)}(f)$ uninteresting.
\end{remark}

\subsection{Background}\label{backsec}

Ideally, one would hope for\footnote{By the above result by Aaronson, if $\int |f| \d\mu =\infty$, which is the case that we are concerned with, then strong laws with $k(N)=0$ are not possible.} $k(N)=constant$, this is referred to as \textit{light trimming}. On the other hand, situations where $k(N)\to\infty$, but $k(N)=o(N)$ as $N\to\infty$ are referred to as \textit{intermediate trimming}.

Trimming is well-studied in the set-up of identically distributed random variables\footnote{Meaning that $f \circ T^n = X_n$ is an independent and identically distributed random process.} (iids). The literature is vast, for brevity's sake, we will only mention situations where explicitly strong laws (or results having immediate consequences for strong laws) are studied. Trimming for iids has been used in other contexts, for example, weak laws, laws of iterated logarithm, CLTs, and more.

Mori (\cite{morilighttrim, mori2}) developed general criteria for lightly trimmed strong laws to hold, his results were later generalized by Kesten and Maller (\cite{kesten1992ratios, km2, mallerstable}). 

From Mori's works (\cite{morilighttrim} gives necessary and sufficient conditions for certain lightly trimmed laws to hold) it already becomes apparent that light trimming is not always sufficient. 
A fortiori Kesten (\cite{Kesten_1993}) showed later that even a lightly trimmed weak law already necessitates an untrimmed weak law (which might not hold). 

A special case, that is of interest, is when\footnote{Here $X$ is a random variable having the same distribution as $f$.} $X$ has regularly varying tails with index $\alpha \in (0,1)$. This means that there is a slowly varying function\footnote{A function $L:[0,\infty) \to\R$ is called \textit{slowly varying} if, for every $c>0$, it holds that $\lim_{x\to\infty} \frac{L(cx)}{L(x)} = 1$.} such that $\P(X>t) = t^{-\alpha} L(t)$. Note that in this situation\footnote{As can be deduced from \cite[VII.7 Theorem 2]{feller1957introduction}.} there is no untrimmed weak law, and hence, by Kesten's result, also no lightly trimmed weak (or strong) law. In this case, a law of the iterated logarithm under trimming was obtained by Haeusler and Mason (\cite{Haeusler_Mason_1987}). Their result was shown to be optimal by Haeusler (\cite{HaeuslernonstandardLIL}). From those results, a strong law under intermediate trimming can be deduced. 

In the context of dynamical systems, so far results have been shown only for systems exhibiting strong mixing properties. Some known results are:
\begin{itemize}
\item If $T$ is the Gauss map and $f(x)=\frac{1}{x}$ (in the case $S_N$ is essentially the sum of coefficients in the CFE\footnote{The continued fraction expansion (CFE) shall be defined later in \S \ref{rotsec}.} of $x$) then it was proved\footnote{They study a situation where the trimming sequence $k(N)$ is allowed to depend on $x$. More explicitly, they show a strong law for
\begin{equation*}
S_N(f)(x) - \theta(N,x) \max_{n=0,...,N-1} f(T^n x) \;\;\; \text{ for some } \; \theta(N,x)\in[0,1].
\end{equation*}
However a statement about $S^1_N$ can be easily deduced from their proof.} by Diamond and Vaaler (\cite{DIAMOND1986}) that $\frac{S_N^1(f)}{N\log(N)}\rightarrow \frac{1}{\log(2)}$. 
\item Even though not formulated in terms of dynamical systems, Aaronson and Nakada (\cite{AARONSONpsi}) showed lightly trimmed strong laws (extending Mori's results) for $\psi$-mixing random variables 
\footnote{For a precise definition of various mixing properties, see \cite{Bradleymixing}.} with speed $\sum_{n=1}^{\infty} \frac{\psi(n)}{n} <\infty$. For Example, Gibbs-Markov maps, together with a function that is measurable for the partition, are exponentially $\psi$-mixing (this follows from \cite{Adpsimix}). Their results apply (among others) if $\P(X>t)=\frac{1}{t}$ for $t>0$. In terms of functions, it applies to $f=\frac{1}{x}$ (with Lebesgue measure).
\item Haynes (\cite{HAYNES_2012}) later showed a quantitative version of the above. Under the same assumptions he showed a strong law for
\begin{equation*}
S_N(f)(x) - \delta(N,x) \max_{n=0,...,N-1} f(T^n x) \;\;\; \text{ for some } \; \delta(N,x)\in\{0,1\},
\end{equation*}
where $(f(T^n x))$ is a $\psi$-mixing process with a certain speed, and provided explicit error terms. We remark that, in contrast to Aaronson and Nakada's result, the number of trimmed terms is allowed to depend on $x$.
\item Kesseb\"ohmer and Schindler (\cite{KESSEBOHMER20194163}) studied intermediate trimming for maps satisfying a spectral gap condition. This method applies to some expanding interval maps (more general than Gibbs-Markov). Their results apply to functions like $f(x)=\frac{1}{x^{\beta}}$ with $\beta>1$ (with Lebesgue measure).
\item Schindler recently proved (\cite{Schindler2018TrimmedSF}) a strong law for the doubling map and the function $f=\frac{1}{x}$. In this situation\footnote{This highlights the influence of periodic points; the function $\frac{1}{x}$ has its singularity at the fixed point $0$.} there is no lightly trimmed strong law as already noted by Haynes (\cite{HAYNES_2012}). Instead, an intermediately trimmed strong law strong law is obtained.
\end{itemize}

\subsection{General results}
Our first result answers Problem A. We show that it is always possible to find a good trimming sequence $k(N)=o(N)$. The following result seems to be known to the experts, for the convenience of the reader we shall provide a proof in \S \ref{gensec}.

\begin{theorem}\label{genthm}
Let $(X, T, \mu)$ be a probability preserving ergodic system. Then, for $f\in \F$ there exist $k(N)\in \N$ with $k(N)=o(N)$ and $d_N>0$ such that 
\begin{equation}\label{genthmclaim}
\lim_{N\rightarrow \infty} \frac{S_N^{k(N)}(f)(x)}{d_N} = 1 \;\;\; \text{for $\mu$-a.e. } x\in X.
\end{equation}
Furthermore, if $(X, T, \mu)$ is uniquely ergodic (and $\mu$ is regular) and $f$ is continuous, then $k(N)$ and $d_N$ can be chosen such that convergence in \eqref{genthmclaim} holds uniformly\footnote{For definiteness, whenever we talk about uniform convergence we replace $f$ by $\hat{f}$ with $\hat{f}=f$ on $\{f<\infty\}$ and $\hat{f}=0$ otherwise.} at all $x\in X$. 
\end{theorem}

Note that Theorem \ref{genthm} does only guarantee the existence of some trimming sequence $k(N)=o(N)$, there is no control whatsoever about how fast $\frac{k(N)}{N}$ decays. In fact, as shall be made more precise later in Remark \ref{seq1/xrem}, for the function $\frac{1}{x}$ over a Liouvillean rotation, $\frac{k(N)}{N}$ might decay arbitrarily slowly.

Heuristically speaking it is most desirable to choose $k(N)$ as small as we can; the smaller $k(N)$ is, the more information about $f$ is retained after trimming $k(N)$ terms. Therefore, Problem A should be reformulated as;

\begin{problem}\label{p2}
What is the "smallest" $k(N)$ such that there are $d_N$ for which \eqref{genthmclaim} holds? Furthermore, give an explicit formula for $d_N$.
\end{problem}

This seems to suggest that if a trimmed strong law holds, then it will continue to hold if we trim more terms. However, this is not true, on the contrary;

\begin{theorem}\label{nonupperstablethm}
Let $(X, T, \mu)$ be an aperiodic\footnote{For a periodic ergodic system, i.e.\ a rotation on finitely many elements, the statement of the Theorem is trivial because every finite function is integrable.} probability preserving ergodic system. Then there is a function $f\in \F$ and a trimming sequence $k(N)=o(N)$ such that there are $d_N>0$ so that \eqref{genthmclaim} holds, but there is another trimming sequence $k'(N)=o(N)$ with $k'(N)\geq k(N)$ such that for all normalising $d'_N>0$ there is an $\epsilon>0$ such that
\begin{equation*}
\limsup_{N\to \infty} \mu\left(\left|\frac{S_N^{k'(N)}(f)}{d'_N}-1\right|>\epsilon\right) >0.
\end{equation*}
\end{theorem}

In the following, we will attempt to answer Problem B in the setting that $T$ is an irrational rotation. This will be done in the next section. 

The proofs of all the theorems of this and the next section will be given in Part \ref{proofpart}.

\section{Trimming for rotations}\label{rotsec}
\subsection{Background}
For irrational $\alpha\in (0,1)$ we consider the rotation $R(x)=R_{\alpha}(x)=x +\alpha$ (mod 1), $x\in [0,1)$ (we identify $\T$ with $[0,1)$ by fixing $0$). First, we shall review some basic facts.

There is a unique\footnote{For rational $\alpha$ the CFE also exists, but is not unique. There are always two representations, namely $\alpha=[a_1,a_2,...,a_n]$ and $\alpha=[a_1,a_2,...,a_n-1,1]$.} representation
\begin{equation}\label{cfedef}
\alpha=\cfrac{1}{a_1 + \cfrac{1}{a_2 + \cfrac{1}{...}}},
\end{equation}
with $a_j\in \N$ for $j\geq 1$, called the \textit{continued fraction expansion} (CFE), we will sometimes also denote \eqref{cfedef} by $\alpha=[a_1, a_2,...]$. In this case $a_1,a_2,...$ shall be referred to as the CFE coefficients. 

If $\alpha=[a_1,a_2,...]$, then $\alpha$ is well approximated by finite iterates of the expansion, i.e.\ by the rational numbers
\begin{equation}\label{qndef}
\frac{p_n}{q_n} = \cfrac{1}{a_1 + \cfrac{1}{a_2 + ... + \cfrac{1}{a_{n-2} + \cfrac{1}{a_{n-1}}}}},
\end{equation}
with the convention $p_1=0$, $q_1=1$. The rational numbers $\frac{p_n}{q_n}$ shall be referred to as the CFE approximants. They are the best rational approximations to $\alpha$ in the sense that
\begin{equation*}
\left|\alpha - \frac{p_n}{q_n}\right| < \left|\alpha - \frac{p}{q}\right| \;\;\;\forall q\leq q_n, p\in \N, (p,q)\not= (p_n,q_n),
\end{equation*}
furthermore, we have
\begin{equation*}
\frac{1}{q_{n+1}(q_{n+1}+q_n)} < \left|\alpha - \frac{p_n}{q_n}\right| < \frac{1}{q_{n+1}^2}.
\end{equation*}
From \eqref{qndef}, it is easy to see that $p_n$ and $q_n$ satisfy the relation
\begin{equation*}
q_{n+1}=a_n q_n +q_{n-1} \; \text{ and } \; p_{n+1}=a_n p_n +p_{n-1} \;\;\; n\geq 2.
\end{equation*}
These basic facts can be found in any book on Diophantine approximations.

For $N\geq 1$ and $n$ such that $N\in [q_n,q_{n+1}-1]$ we can uniquely expand $N$ as
\begin{equation}\label{ostrdef}
N=\sum_{j=1}^n b_j q_j,
\end{equation}
with $b_j\in \N$ and $b_j \leq a_j \leq  \frac{q_{j+1}}{q_j}$, 
the expansion in \eqref{ostrdef} is unique if we require that, for each $J=1,...,n-1$, we have\footnote{Equivalently $b_1\not=a_1$ and there is no $j\in [2,n]$ with $b_j=a_j$ and $b_{j-1} \geq 1$.} $\sum_{j=1}^J b_j q_j < q_{J+1}$. This is called the Ostrowski expansion, note that $b_n=\ffloor{N}{q_n}$. In an attempt to keep the notation simple we will, in the following, always implicitly assume that $N\in [q_n,q_{n+1}-1]$ and \eqref{ostrdef} holds, thus suppressing the dependence of $n$ and $b_j$ on $N$.

We will often make use of the Denjoy-Koksma inequality; for $f\in BV$ and $n\geq 1$ it holds that
\begin{equation*}
\Big|S_{q_n}(f)(x) - q_n \int_0^1 f \d\lambda\Big| \leq  \Var(f) \;\;\;\forall x\in [0,1),
\end{equation*}
where $\lambda$ denotes the Lebesgue measure on $[0,1)$ and $\Var(f)$ denotes the total variation of $f$. In light of the Ostrowski expansion \eqref{ostrdef}, it follows that
\begin{equation*}
\Big|S_{N}(f)(x) - N \int_0^1 f \d\lambda\Big| \leq \sum_{j=1}^n b_j \Var(f) \;\;\;\forall x\in [0,1).
\end{equation*}

For simplicity's sake, in this section, we will focus on the functions
\begin{equation*}
f(x)=x^{-\beta}, \; \beta\geq 1.
\end{equation*}

Ergodic sums of $x^{-\beta}$ - or the "zero average" version $x^{-\beta}-(1-x)^{-\beta}$ - have already been studied without trimming, among others in \cite{dolgopyat2020limittheoremstoraltranslations}, \cite{sinai2008limittheorembirkoffsums}. However, there the rotation number $\alpha$ is also randomised and one obtains a distributional limit, to a non-constant distribution, rather than convergence to a constant. For example \cite[Theorem 2]{sinai2008limittheorembirkoffsums} shows; For $f(x)=\frac{1}{x}-\frac{1}{1-x}$ there is a distribution $\mathcal{D}$ on $\R$ such that
\begin{equation*}
\lambda^2\left( (\alpha,x) \;\left|\; a\leq \frac{1}{N} S_N(f)(x)\leq b\right. \right) \rightarrow \mathcal{D}[a,b] \;\;\;\as{N} \quad\forall a,b \in \R,
\end{equation*}
where $\lambda^2$ denotes the Lebesgue measure on $\T^2$. In \cite{dolgopyat2020limittheoremstoraltranslations} an analogous result is obtained for $f(x)=x^{-\beta}$ with $\beta>1$. By linearity and symmetry their result can be generalised to $f(x)=c_1 x^{-\beta} + c_2 (1-x)^{-\beta}$ for any $c_1,c_2\in \R$. Usually, the case $c_1=-c_2$ is referred to as the \textit{symmetric} case, whereas $c_1\not= - c_2$ is called the \textit{asymmetric} case.

This highlights, that the strength of our results is that $\alpha$ can be fixed. The price we have to pay is that trimming might be necessary.

The following theorems of this section are stated for the functions $\frac{1}{x}$ resp.\ $x^{-\beta}$. More generally, they apply also to $\frac{c_1}{x} + \frac{c_2}{1-x}$ resp.\ $c_1 x^{-\beta} + c_2 (1-x)^{-\beta}$ with $c_1\not=-c_2$, with only minor changes in the proof which we leave to the interested reader. However, note that, in contrast to \cite{dolgopyat2020limittheoremstoraltranslations} and \cite{sinai2008limittheorembirkoffsums}, we only discuss the asymmetric case.

\subsection{The function $\frac{1}{x}$}
First let us consider $f(x)=\frac{1}{x}$. As it turns out, for almost all $\alpha$ and the function $f$, a weak law of large numbers holds, even without trimming. On the other hand, for a $G_{\delta}$-dense set of $\alpha$, even the weak law does not hold without trimming.

The crucial condition here is a well-known Diophantine condition on $\alpha$.

\begin{definition}
The number $\alpha$ is said to be of \textit{Roth type} if, for all $\epsilon>0$ and large enough $n$, it holds that $q_{n+1}<q_n^{1+\epsilon}$. Equivalently
\begin{equation*}
\lim_{n\to \infty} \frac{\log(q_{n+1})}{\log(q_n)}=1.
\end{equation*}
\end{definition} 

\begin{theorem}\label{weak1/xthm}
There are $d_N>0$ such that
\begin{equation*}%\label{weak1/xthmrotcon}
\lambda\left(\left|\frac{S_N(f)}{d_N}-1\right|>\epsilon\right) \to 0 \;\;\;\as{N} \quad \forall \epsilon>0,
\end{equation*}
if and only if $\alpha$ is of Roth type. Furthermore, in this situation 
\begin{equation*}
d_N= N\log(N).
\end{equation*}
\end{theorem}

\begin{remark}\label{seq1/xrem}
Contrary to the situation in Theorem \ref{weak1/xthm}, for $\alpha$ not being of Roth type, we may need a trimming sequence arbitrarily close to $N$ to obtain a trimmed weak law. More precisely, given any sequence $l(N)=o(N)$, there is a $G_{\delta}$ dense set of $\alpha$ such that, whenever $k(N)\leq l(N)$ and $d_N>0$, there is an $\epsilon>0$ such that
\begin{equation*}%\label{seq1/xconv}
\limsup_{N\to \infty} \lambda\left(\left|\frac{S_N^{k(N)}(f)}{d_N}-1\right|>\epsilon\right) >0.
\end{equation*}
\end{remark}

\begin{remark}\label{weakgenrem}
A trimmed weak law generally implies an untrimmed weak law under condition \eqref{weakgencond}. More precisely, let $(X, T, \mu)$ be a probability-preserving dynamical system, and suppose there exist $k(N) = o(N)$ and $d_N > 0$ such that
\begin{equation*}
\mu\left(\left|\frac{S_N^{k(N)}(f)}{d_N}-1\right|>\epsilon\right) \to 0 \quad \text{as } N\to\infty,  \quad\forall \epsilon>0.
\end{equation*}
If, in addition,
\begin{equation}\label{weakgencond}
\mu\left( f > c \frac{d_N}{k(N)} \right) = o\left(\frac{1}{N}\right) \quad \forall c>0,
\end{equation}
then the untrimmed weak law follows:\footnote{For any $\epsilon > 0$, we estimate
\begin{align*}
\mu\left( \left| S_N(f) - d_N\right| > \epsilon d_N \right) 
& \leq \mu\left( \left| S_N^{k(N)}(f) - d_N\right| > \frac{\epsilon d_N}{2} \right) + \mu \left( \exists 0\leq n\leq N-1 \;|\; f\circ T^n > \frac{\epsilon d_N}{2 k(N)} \right) \\
& \leq \mu\left( \left| S_N^{k(N)}(f) - d_N \right| > \frac{\epsilon d_N}{2} \right) + N \mu\left( f > \frac{\epsilon d_N}{2 k(N)} \right) \to 0,
\end{align*}
as $N\to\infty$.}
\begin{equation*}
\mu\left(\left|\frac{S_N(f)}{d_N}-1\right|>\epsilon\right) \to 0 \quad \text{as } N\to\infty, \quad \forall \epsilon>0.
\end{equation*}
\end{remark}

Next, we investigate strong laws. As above for iids, or the Gauss map, for almost all $\alpha$ it suffices to trim by $k(N)=1$.

\begin{theorem}\label{strong1/xthm}
For almost all $\alpha$ it holds that 
\begin{equation}\label{strong1/xthmcon}
\lim_{N\to \infty} \frac{S_N^1(f)(x)}{N \log(N)} = 1 \;\;\; \text{ for a.e. } x\in [0,1).
\end{equation}
Furthermore, if $\alpha$ is of bounded type, then this convergence is uniform in $x$, i.e.
\begin{equation*}
\lim_{N\to \infty} \frac{S_N^1(f)(x)}{N \log(N)} = 1 \;\;\; \text{ uniformly in } x\in [0,1).
\end{equation*}
\end{theorem}

The convergence in \eqref{strong1/xthmcon} only holds for almost every $x$, in fact;

\begin{theorem}\label{nonconat0thm1/x}
For almost all $\alpha$ it holds that
\begin{equation*}
\limsup_{N\to \infty} \frac{S_N^1(f)(\alpha)}{N \log(N)} > 1 \geq \liminf_{N\to \infty} \frac{S_N^1(f)(\alpha)}{N \log(N)}.
\end{equation*}
\end{theorem}

\begin{remark}
The set of all $x$ such that 
\begin{equation*}
\limsup_{N\to \infty} \frac{S_N^1(f)(x)}{N \log(N)} > 1 \geq \liminf_{N\to \infty} \frac{S_N^1(f)(x)}{N \log(N)}
\end{equation*}
is invariant under $R$, as will be shown in the proof of Theorem \ref{noncothm1/x}. Therefore it is also true that
\begin{equation*}
\limsup_{N\to \infty} \frac{S_N^1(f)(0)}{N \log(N)} > 1 \geq \liminf_{N\to \infty} \frac{S_N^1(f)(0)}{N \log(N)}.
\end{equation*}
\end{remark}

Lastly, there are\footnote{As before, the set of such $\alpha$ will be $G_{\delta}$-dense. We will not show this fact, as it can easily be verified and is not essential to this work, leaving it instead to the reader.} $\alpha$, such that the untrimmed weak law holds\footnote{Or equivalently $\alpha$ is of Roth type.}, but not the trimmed strong law for any $k(N)$ being constant.

\begin{theorem}\label{noncothm1/x}
There is an $\alpha$, such that 
\begin{equation*}
\lambda\left(\left|\frac{S_N(f)}{N \log(N)}-1\right|>\epsilon\right) \to 0 \;\;\;\as{N} \quad \forall \epsilon>0,
\end{equation*}
but, for every $K\geq 1$ and almost every $x$, it holds that
\begin{equation}\label{noncoclaim1/x}
\limsup_{N\to \infty} \frac{S_N^K(f)(x)}{N \log(N)} > 1 \geq \liminf_{N\to \infty} \frac{S_N^K(f)(x)}{N \log(N)}.
\end{equation}
\end{theorem}

\subsection{The function $x^{-\beta}$ for $\beta>1$}

Let $\beta>1$ and $f(x)=x^{-\beta}$. We will prove sufficient and necessary conditions for $k(N)$ such that trimmed weak or strong laws hold. The general conditions are slightly complicated to state, but, in the case of monotone trimming sequences, reduce to a much simpler one. Here we will only state the condition for monotone trimming sequences, leaving the more general statements for the proof section.

\begin{theorem}\label{betathm}
Let $k(N)$ be monotone with $k(N)=o(N)$ and $d_N>0$, then the following are equivalent

(I) the trimmed strong law holds uniformly, i.e.
\begin{equation}\label{betastrongcon}
\lim_{N\to \infty} \frac{S_N^{k(N)}(f)(x)}{d_N} = 1 \;\;\; \text{ uniformly in } x\in [0,1),
\end{equation}
(II) the trimmed weak law holds, i.e.
\begin{equation}\label{betaweakcon}
\lambda\left(\left|\frac{S_N^{k(N)}(f)}{d_N}-1\right|>\epsilon\right) \to 0 \text{ as } N\to\infty
\quad \forall \epsilon>0,
\end{equation}
(III)\begin{equation}\label{betacond}
\frac{k(N)}{\max (b_n , \max_{j\leq n-1} a_j)} \to \infty\text{ as } N\to\infty.
\end{equation}
Furthermore, in this situation, $d_N=\frac{1}{\beta - 1} N^{\beta} k(N)^{1-\beta}$.
\end{theorem}

\subsection{Comparison of results to literature}\label{compsec}

As our results seem to be the first that do not require strong mixing properties of the system, it is of special interest to compare our results, to the known ones listed in \S \ref{backsec}. We shall compare trimming results for the non-integrable functions $\frac{1}{x}$ and $x^{-\beta}$ with $\beta>1$. All of the results below were discussed in \S\ref{backsec} or earlier in this section. In each cell of the table, we write sufficient conditions for a trimming sequence $k(N)=o(N)$ such that there are $d_N>0$ with
\begin{equation*}
\lim_{N\rightarrow \infty} \frac{S_N^{k(N)}(f)(x)}{d_N} = 1 \;\;\; \text{for $\mu$-a.e. } x\in X.
\end{equation*}

To simplify notation, from now on we introduce the following; For $x$ sufficiently large such that the below is defined denote 
\begin{equation*}
\log_1(x)=\log(x) \; \text{ and } \; \log_k(x)= \log(\log_{k-1}(x)), \;\; k\geq 2.
\end{equation*}

% \centerline{
% \begin{tabular}{c|c|c}
%                      &$\frac{1}{x}$                           &$x^{-\beta}$, $\beta>1$\\
% \hline
% iid                  &$k(N)=1$                                &$\frac{k(N)}{\log_2(N))} \to \infty$\\
% \hline
% Gauss map            &$k(N)=1$                                &$\frac{k(N)}{\log_2(N))} \to \infty$\\
% \hline
% doubling map         &$k(N)=\lceil \kappa \log_3(N)\rceil$, $\kappa>\frac{1}{\log(2)}$&$\frac{k(N)}{\log_2(N))} \to \infty$\\
% \hline
% almost every rotation&$k(N)=1$                                &$\frac{k(N)}{\log^{1+\epsilon}(N)} \rightarrow \infty$ 
% \end{tabular}
% }

\begin{table}[ht]
\renewcommand*{\arraystretch}{1.7}
\centerline{
\begin{tabu}{c|[1.5pt]c|c}
                    &$\frac{1}{x}$                          &$x^{-\beta}$, $\beta>1$\\
\tabucline[1.5pt]{-}
iid                  &$k(N)=1$                                &$\frac{k(N)}{\log_2(N)} \to \infty$\\
\hline
Gauss map            &$k(N)=1$                                &$\frac{k(N)}{\log_2(N)} \to \infty$\\
\hline
doubling map         &$k(N)=\lceil \kappa \log_3(N)\rceil$, $\kappa>\frac{1}{\log(2)}$&$\frac{k(N)}{\log_2(N)} \to \infty$\\
\hline
almost every rotation&$k(N)=1$                                &$\frac{k(N)}{\log^{1+\epsilon}(N)} \rightarrow \infty$ 
\end{tabu}
}
\end{table}

In all of the cases we have $d_N=c_1 N \log(N)$ for the function $\frac{1}{x}$ and $d_N=  \frac{c_{\beta}}{\beta-1} N^{\beta} k(N)^{-\beta+1}$ for $x^{-\beta}$, where $c_1=c_{\beta}=1$ in all cases except for the Gauss map - in case of the Gauss map we have $c_1=\frac{1}{\log(2)}$ and $c_{\beta}=\frac{1}{(\log (2))^{\beta}}$\footnote{The phenomenon that we get a different constant in case of the Gauss map is due to the fact the Gauss map is invariant with respect to the Gauss measure which is equivalent but not equal to the Lebesgue measure - all other ergodic transformations are invariant with respect to the Lebesgue measure.}.

A similar comparison can be made for trimmed weak laws;
\begin{itemize}
\item For iids, the Gauss map, and the doubling map, together with the function $\frac{1}{x}$ weak laws hold even without trimming. This was shown in \cite[Theorem 2 VII.]{feller1957introduction}, \cite{khintchine1935} and \cite{Schindler2018TrimmedSF} respectively, but can be also deduced directly from the trimmed strong laws in the table above.\footnote{This follows from Remark \ref{weakgenrem} since in all cases, for all $c>0$,
\begin{equation*}
\mu\left(f>c\frac{d_N}{k(N)}\right) \ll \frac{k(N)}{N \log(N)} = o\left(\frac{1}{N}\right).
\end{equation*}}
\item For iids, the Gauss map, and the doubling map, and the function $x^{-\beta}$ trimmed weak laws hold for any intermediate trimming sequence $k(N)\to \infty$ with $k(N)=o(N)$. This was shown in \cite{kessschimean}.
\item For rotations of Roth type and the function $\frac{1}{x}$ the weak law also holds without trimming, as shown in Theorem \ref{weak1/xthm}. However, as demonstrated in Remark \ref{seq1/xrem}, for a generic rotation number $\alpha$, we might require a trimming sequence arbitrarily close to $N$ to obtain a trimmed weak law. 
\item For rotations and the function $x^{-\beta}$, Theorem \ref{betathm} states exact conditions on $k(N)$ so that a trimmed weak law holds. Also here generically we will need a sequence arbitrarily close to $N$. Only for $\alpha$ of bounded type we can choose a trimming sequence that approaches infinity arbitrarily slowly.
\end{itemize}

This suggests, that trimming is well applicable even beyond the mixing case. Finding more flexible conditions and techniques to obtain effective trimming results without using strong mixing properties is a promising direction for future research. 

\part{Proofs}\label{proofpart}

\section{General trimmed laws}\label{gensec}

\begin{proof}[Proof of Theorem \ref{genthm}]
(i) For every $l\geq 1$ denote $f_l=\min(f,l)$, and $a_l = \int_X f_l \d\mu$. Using Jegorow's Theorem, there are sets $B_l$ with $\mu(B_l)<2^{-l}$ and $1\leq N_1\leq N_2\leq ...$, with $N_l\geq l$ such that
\begin{equation}\label{genthmmueq}
|S_n(1_{\{f \geq l\}})(x) - n \mu(f\geq l)| < 2^{-l} n \mu(f\geq l)
\end{equation}
and
\begin{equation}\label{genthmfleq}
|S_n(f_l)(x) - n a_l| < 2^{-l} n a_l
\end{equation}
for $x\not\in B_l$ and $n\geq N_l$. Let $k(n)$ be defined as
\begin{equation*}
k(n)=\lceil (1+2^{-l}) n \mu(f\geq l) \rceil, \;\;\; n\in [N_l,N_{l+1}-1].
\end{equation*}
Since $\mu(f\geq l) \to 0$ as $l\to\infty$, it is clear that $k(n)=o(n)$.
By the Borel-Cantelli lemma, almost every $x$ is only in finitely many $B_l$. For such an $x$ and $l$ big enough such that $x\not \in B_j$ for $j\geq l$, using \eqref{genthmmueq} and \eqref{genthmfleq}, for $n\geq N_l$ it holds that
\begin{align*}
|S_n^{k(n)}(f)(x) - n a_l| 
&\leq |S_n^{k(n)}(f)(x)-S_n(f_l)(x)| + |S_n(f_l)(x) - n a_l|\\
& < 2^{-l+1} l n \mu(f\geq l) + l + 2^{-l} n a_l.
\end{align*}
Since $a_l\nearrow \infty$ as $l\rightarrow \infty$, and $n\geq l$, the right-hand side is $o(n a_l)$, hence 
\begin{equation*}
\lim_{n\rightarrow \infty} \frac{S_n^{k(n)}(f)(x)}{d_n} = 1  \;\;\; \text{for $\mu$-a.e. } x\in X,
\end{equation*}
where $d_n=n a_l$ for $n\in [N_l,N_{l+1}-1]$.

(ii) Now suppose $(X,T,\mu)$ is uniquely ergodic (and $\mu$ is regular) and $f$ is continuous. 
Since the sets $\{f=s\}$, for $s\geq 0$, are disjoint, only countably many can have positive measure. Therefore, let $s_l\nearrow \infty$ be a sequence with $\mu(f=s_l)=0$, let $f_l=\min(f,s_l)$ and $a_l=\int_X f_l \d \mu$, note that $f_l$ is continuous. By regularity of $\mu$, there is an open set $O_l\supset \{f\geq s_l\}$ such that $\mu(O_l \setminus \{f\geq s_l\}) < 2^{-l} \mu(f\geq s_l)$, and by Urysohn's Lemma there is a continuous function $\chi^+_l$ that is $1$ on $\{f\geq s_l\}$ and $0$ outside $O_l$. Likewise, there is a closed set $C_l\subset \{f>s_l\}$, with $\mu(\{f>s_l\} \setminus C_l) < 2^{-l} \mu(f\geq s_l)$, and a continuous function $\chi^-_l$ that is $1$ on $C_l$ and $0$ outside $\{f>s_l\}$. There are $1\leq N_1 \leq N_2 \leq ...$, with $N_l\geq s_l$, such that
\begin{align*}
\Big|S_n(\chi^+_l)(x) - n \int_X \chi^+_l \d\mu \Big| &< 2^{-l} n \mu(f\geq s_l) \; \text{ and} \\
 \Big|S_n(\chi^-_l)(x) - n \int_X \chi^-_l \d\mu \Big| &< 2^{-l} n \mu(f\geq s_l)
\end{align*}
as well as
\begin{equation*}
|S_n(f_l)(x) - n a_l| < 2^{-l} n a_l,
\end{equation*}
for $x\in X$ and $n\geq N_l$. For
\begin{equation*}
k(n)=\lceil (1+2^{-l+1}) n \mu(f\geq s_l) \rceil, \;\;\; n\in [N_l,N_{l+1}-1],
\end{equation*}
analogously to the above, it follows that
\begin{align*}
|S_n^{k(n)}(f)(x) & - n a_l| < 2^{-l+2} s_l n \mu(f\geq s_l) + s_l + 2^{-l} n a_l.
\end{align*}
Since $a_l\geq s_l \mu(f\geq s_l)$, $a_l \to \infty$ and $n\geq N_l \geq s_l$, the right-hand side is $o(n a_l)$ as $l\to\infty$. The claim follows with $d_n=n a_l$ for $n\in [N_l,N_{l+1}-1]$.
\end{proof}

To prove Theorem \ref{nonupperstablethm}, we will construct a function $f$ using Rokhlin Towers. Here we use the following version of Rokhlin's Tower theorem (\cite[Theorem 1.5.9]{aaronson1997introduction}). 

\begin{lemma}\label{rolem}
Let $(X, T, \mu)$ be an aperiodic probability preserving ergodic system. Then, for $N\geq 1$ and $\epsilon>0$, there is a measurable set $A$ such that $\{T^{-j}(A)\}_{j=0}^{N-1}$ are disjoint and $\mu\left(X \setminus \bigcup_{j=0}^{N-1} T^{-j}(A) \right) <\epsilon$.
\end{lemma}

\begin{proof}[Proof of Theorem \ref{nonupperstablethm}]
In order to prove Theorem \ref{nonupperstablethm} it will be enough to construct a function $f\in \mathcal{F}$, a sequence $\s{N}{l}$, and normalising sequences $D_l,D'_l>0$ such that 
\begin{equation}\label{nonustrong}
\lim_{l\rightarrow \infty} \frac{S_{N_l}(f)(x)}{D_l} = 1 \;\;\; \text{for $\mu$-a.e. } x\in X,
\end{equation}
but
\begin{equation}\label{nonuweak}
\mu\left(S_{N_l}^1(f)\geq D'_l\right) \geq \frac{1}{3}, \; \text{ while } \; \mu\left(S_{N_l}^1(f)\leq \frac{D'_l}{10}\right) \geq \frac{1}{3}.
\end{equation}
To conclude the proof, let $\tilde{k}(N)=o(N)$ and $\tilde{d}_N>0$ be such that
\begin{equation*}
\lim_{N\rightarrow \infty} \frac{S_{N}^{\tilde{k(N)}}(f)(x)}{\tilde{d}_N} = 1 \;\;\; \text{for $\mu$-a.e. } x\in X,
\end{equation*}
the existence of such $\tilde{k}(N)$, $\tilde{d}_N$ is guaranteed by Theorem \ref{genthm}. Now the conclusion follows for
\begin{equation*}
k(N)=\begin{cases}
0 & \text{if } N=N_l \text{ for some } l\geq 1,\\
\tilde{k}(N) & \text{otherwise},
\end{cases}
\end{equation*}
\begin{equation*}
k'(N)=\begin{cases}
1 & \text{if } N=N_l \text{ for some } l\geq 1,\\
\tilde{k}(N) & \text{otherwise},
\end{cases}
\end{equation*}
\begin{equation*}
d_N=\begin{cases}
D_l & \text{if } N=N_l \text{ for some } l\geq 1,\\
\tilde{d}_N & \text{otherwise},
\end{cases}
\end{equation*}
and
\begin{equation*}
d'_N=\begin{cases}
D'_l & \text{if } N=N_l \text{ for some } l\geq 1,\\
\tilde{d}_N & \text{otherwise}.
\end{cases}
\end{equation*}

Let $N_l=2^{2^{2^l}}$ and $\epsilon_l=2^{-N_l}$, the Rokhlin Tower theorem, Lemma \ref{rolem}, yields a measurable set $A_l$ such that $\{T^{-j}(A_l)\}_{j=0}^{N_l-1}$ are disjoint and $\mu(R_l)<\epsilon_l$, where $R_l=X\setminus \bigcup_{j=0}^{N_l - 1} T^{-j}(A_l)$. Note that $\frac{1-2^{-N_l}}{N_l}\leq \mu(A_l) \leq \frac{1}{N_l}$. Let $B_l\subset T^{-1}(A_l)$ be a measurable set with $\mu(B_l)=\frac{\mu(A_l)}{2}$ and define a function $g_l:X\to [0,\infty)$ by 
\begin{equation*}
g_l(x)=\begin{cases}
2^{N_l} & \text{if } x\in A_l,\\
2^{N_l-l}& \text{if } x\in B_l,\\
0& \text{otherwise}.
\end{cases}
\end{equation*}
Let $f=\sum_{l=1}^{\infty} g_l$. Since each $g_l$ is non-zero only on a set of measure at most $\frac{3}{2N_l}$, the function $f$ is almost everywhere defined by a finite sum. Furthermore, $f$ is non-integrable since $\int g_l \d\mu \geq \frac{1}{2}$ for each $l\geq 1$. 

For each $l\geq 1$, every point $x\not\in R_l$ will visit $A_l$ within $N_l$ steps, therefore $S_{N_l}(f)(x) \geq 2^{N_l}$. On the other hand, if $x$ does not visit $A_k\cup B_k$ within\footnote{Here we count $x$ itself as the first step.} $N_l$ steps for any $k>l$, then, for big enough $l$, it holds
\begin{equation*}
S_{N_l}(f)(x) \leq (1+2^{-l}) 2^{N_l} + N_l \sum_{j=1}^{l-1} ||g_j||_{L^{\infty}} \leq (1+2^{-l}) 2^{N_l} + N_l \sum_{j=1}^{l-1} 2^{N_j} \leq (1+2^{-l+1}) 2^{N_l}.
\end{equation*} 
Therefore, for big enough $l$,
\begin{align*}
\mu\left( \left| \frac{S_{N_l}(f)}{2^{N_l}} - 1 \right| > 2^{-l+1}\right) &\leq \mu\left( R_l \cup \tilde{R}_l\right)
\leq 2^{-N_l} + 2 N_l \sum_{k=l+1}^{\infty} \frac{1}{N_k} \leq \frac{1}{N_l},
\end{align*}
where $\tilde{R_l}=\bigcup_{k=l+1}^{\infty} \bigcup_{j=0}^{N_l-1} T^{-j}(A_k\cup B_k)$. By Borel-Cantelli \eqref{nonustrong} holds with $D_l=2^{N_l}$.

On the other hand, every point $x\not\in A_l \cup R_l$ will visit $T^{-1}(A_l)$ within $N_l$ steps, of those points half will visit $B_l$ and the other half will visit\footnote{And hence they will avoid $B_l$ altogether since it would take at least $N_l$ steps to visit $T^{-1}(A_l)$ again.} $T^{-1}(A_l) \setminus B_l$. In addition, all of these points will visit $A_l$ exactly once, that is the largest value that we trim, unless $A_k \cup B_k$ is visited for some $k>l$, then the largest value will be one of the terms contributed by $g_k$. Therefore, for big enough $l$, we have
\begin{equation*}
\mu\left( S_{N_l}^1(f) \geq 2^{N_l-l}\right) \geq \mu\left( \bigcup_{j=0}^{N_l - 2} T^{-j}(B_l) \right) \geq \frac{(1-2^{-N_l})(N_l - 1)}{2 N_l}\geq \frac{1}{3}.
\end{equation*}
If $x$ does not visit $B_l$ within $N_l$ steps, and avoids all $A_k \cup B_k$ for $k>l$, then all the non-zero contribution can only come from visiting $A_l$, which however is the value that will be trimmed, and visits to $A_k \cup B_k$ for $k<l$. Therefore, for $x\in \bigcup_{j=0}^{N_l - 2} T^{-j}(T^{-1}(A_l)\setminus B_l) \setminus \tilde{R}_l$ and big enough $l$ it holds that
\begin{align*}
S_{N_l}^1(f)(x) \leq N_l \sum_{k=1}^{l-1} ||g_k||_{L^{\infty}} \leq N_l \sum_{k=1}^{l-1} 2^{N_k} \leq \frac{1}{10} 2^{N_l-l}.
\end{align*}
The set in question has measure
\begin{equation*}
\mu\left( \bigcup_{j=0}^{N_l - 2} T^{-j}(T^{-1}(A_l)\setminus B_l) \setminus \tilde{R}_l \right) \geq \frac{(1-2^{-N_l})(N_l - 1)}{2 N_l} - 2 N_l \sum_{k=l+1}^{\infty} \frac{1}{N_k} \geq \frac{1}{3}.
\end{equation*}
It follows that \eqref{nonuweak} holds with $D'_l=2^{N_l-l}$. 
\end{proof}

\section{Irrational rotations on $\T$}

\subsection{Background}

Recall that, for $N\geq 1$ and $n$ such that $N \in [q_n,q_{n+1} - 1]$, there are $0\leq b_j < \frac{q_{j+1}}{q_j}$, for $j=0,...,n$, such that 
\begin{equation*}
N=\sum_{j=0}^n b_j q_j,
\end{equation*}
this is called the Ostrowski expansion.

In the following we will, by slight abuse of notation, identify $[0,1)$ with $\left[ -\frac{1}{2}, \frac{1}{2} \right)$ via the identification $\iota:[0,1) \rightarrow \left[ -\frac{1}{2}, \frac{1}{2} \right)$ 
\begin{equation*}
\iota(x)=\begin{cases}
x & \text{if } \; x < \frac{1}{2},\\
x-1 & \text{otherwise}.
\end{cases}
\end{equation*}
Whenever we write $x<0<y$ for some points $x,y\in [0,1)$, we implicitly refer to this identification.

Denote 
\begin{equation}\label{DefDelta}
\delta_n = |\alpha q_{n-1}|.
\end{equation}
Then
\begin{equation*}
\delta_n\in \left(\frac{1}{2q_n},\frac{1}{q_n}\right). 
\end{equation*}
Furthermore, as is easily deduced by the approximation properties of $\alpha$, $\delta_n$ is given by a simple recursion.

\begin{lemma}\label{rotdeltalem}
For all $n\geq 1$ it holds that $a_n \delta_{n+1} + \delta_{n+2} = \delta_n $.
\end{lemma}
\begin{proof}
Assume $\alpha-\frac{p_n}{q_n}>0$, the other case can be proven similarly. Since $\alpha-\frac{p_{n-1}}{q_{n-1}}<0$ and $\alpha-\frac{p_{n+1}}{q_{n+1}}<0$, we have 
\begin{equation*}
R^{q_{n-1}}(0)< R^{q_{n-1}+q_n}(0) < ... < R^{q_{n-1} + a_n q_n}(0)=R^{q_{n+1}}(0) < 0.
\end{equation*}
The claim follows since $d(R^{q_{n-1}}(0),0)=\delta_n, \; d(R^{q_{n+1}}(0), 0)=\delta_{n+2} $ and 
\begin{equation*}
d(R^{q_{n-1}+(i-1) q_n}(0), R^{q_{n-1}+i q_n}(0))=\delta_{n+1}\quad \forall i=1,...,a_n. 
\end{equation*}
\end{proof}

For $N\geq 1$ and $x\in [0,1)$ denote by $j_1^N(x),...,j_N^N(x)$ the clockwise ordering of the points $\{x,...,R^{N-1}(x)\}$ around the circle, for convenience leaving out $0$ so that
\begin{equation*}
0< R^{j_1^N(x)}(x)<...<R^{j_N^N(x)}(x)\leq 1,
\end{equation*}
where we identify $\T$ with $(0,1]$, and
\begin{equation*}
f(R^{j_1^N(x)}(x))>...>f(R^{j_N^N(x)}(x)),
\end{equation*}
where we recall that, by the convention in Theorem \ref{genthm}, we set $f(0)=0$.
Sometimes we will also denote $R^{j_1^N(x)}(x)=x^N_{\min}$.

It holds that
\begin{equation*}
f(R^{j_1^N(x)}(x))>...>f(R^{j_N^N(x)}(x)),
\end{equation*}
where $f(x)=x^{-\beta}$ for some $\beta\geq 1$. Hence, for $N\geq 1$ and $k\leq N$, it holds that
\begin{equation*}
S_N^k(f)(x) = \sum_{l=k+1}^N f(R^{j_l^N(x)}(x)).
\end{equation*}
From this, we can immediately deduce some helpful bounds. Consider the case when $N\in [q_n,q_{n+1} - 1]$, $\alpha - \frac{p_n}{q_n} > 0$ and $j_1^N(x)=j_1^{q_n}(x)$. Then the points $\{x,...,R^{N-1}(x)\}$ are grouped into clusters of size $b_n$ or $b_n + 1$. If in addition $R^{j_1^N(x)}(x)>\frac{\epsilon}{q_n} $, for some $\epsilon>0$, then, for each $k\geq 0$, we can bound
\begin{equation*}
S^k_N(f)(x) \leq 2^{\beta} q_n^{\beta} (b_n + 1) \sum_{j=\ffloor{k}{b_n}}^{q_n-1} (\epsilon + j)^{-\beta}.
\end{equation*}
This estimate is obtained by noticing that the leftmost point of each group\footnote{These are the points $\{x,..., R^{q_n-1}(x)\}$ numbered clockwise, starting with $R^{j_1^{q_n}(x)}(x)$.} is at distance at least $\delta_n\geq \frac{1}{2q_n}$.
A similar estimate can be made, assuming $k\geq b_n + 1$ instead of $j_1^N(x)=j_1^{q_n}(x)$, but the point of the above is that this calculation is independent of $k$. Many of our proofs will rely on this or a similar estimate, therefore it will be crucial to first understand better when $j_1^N(x)=j_1^{q_n}(x)$.

\begin{lemma}\label{j1lem}
If $\alpha - \frac{p_n}{q_n}>0$ and\footnote{This is to be understood in $[-\frac{1}{2},\frac{1}{2})$.} $R^{j_{q_n}^{q_n}(x)}(x) \leq - b_n \delta_{n+1}$, then $j_1^N(x)=j_1^{q_n}(x)$.
\end{lemma}
\begin{proof}
It holds that
\begin{equation*}
R^{j_{q_n}^{q_n}(x)}(x)<R^{j_{q_n}^{q_n}(x)+q_n}(x)<...< R^{j_{q_n}^{q_n}(x)+b_n q_n}(x) < R^{j_1^{q_n}(x)}(x),
\end{equation*}
and no other point of $\{x,...,R^{N-1}(x)\}$ is between $R^{j_{q_n}^{q_n}(x)}(x)$ and $R^{j_1^{q_n}(x)}(x)$. Furthermore,
\begin{equation*}
R^{j_{q_n}^{q_n}(x)+b_n q_n}(x) \leq R^{j_{q_n}^{q_n}(x)}(x) + b_n \delta_{n+1} \leq 0,
\end{equation*}
and the conclusion is immediate.
\end{proof}

\subsection{The function $\frac{1}{x}$}
For this section, denote $f(x)=\frac{1}{x}$.

Before starting the section with some technical results, we first give an overview of the main propositions in this section. First, using the Denjoy-Koksma inequality, we find suitable $d_N>0$ and $\epsilon_n>0$ such that 
\begin{equation*}
|S_N^{b_n+1}(f)(x)-d_N| <\epsilon_N \quad \forall x\in [0,1),
\end{equation*} 
this is shown in Lemma \ref{errlem} and \ref{strong1/xlembn}. If $\alpha$ is of Roth type, then we can show that $d_N\sim N\log(N)$ and $\epsilon_N=o(N \log(N))$, this will be shown in Proposition \ref{logprop} and Lemma \ref{err1/xbnlem} respectively. This leads to the intermediate result
\begin{equation}\label{snbncon1/x}
\lim_{N\to\infty}\frac{S_N^{b_n + 1}(f)(x)}{N\log(N)} = 1 \quad \text{uniformly in } x\in [0,1),
\end{equation}
for $\alpha$ of Roth type, as stated in Proposition \ref{strong1/xprop}.

In order to prove Theorems \ref{weak1/xthm} and \ref{strong1/xthm}, we will in both cases compare $S_N^{b_n + 1}(f)$ to $S_N^{1}(f)$ and then make use of the convergence \eqref{snbncon1/x}. In the case of  Theorem \ref{weak1/xthm} we note that, as in \S \ref{compsec}, a lightly trimmed weak law for the function $\frac{1}{x}$ already implies an untrimmed weak law.

In the last part of this section we will prove Theorems \ref{nonconat0thm1/x} and \ref{noncothm1/x} making use of the previous results to prove the $\liminf$ case and construct some sets for which the $\limsup$ case holds.

\begin{lemma}\label{countlem}
For $n\geq 1$ and $x\in [0,1)$ there is at most one $k^*\in \{1,...,q_n\}$ such that
\begin{equation*}
R^{k^*}(x)\in \left[ 0,\frac{1}{q_n+q_{n-1}}\right).
\end{equation*}
\end{lemma}
\begin{proof}
For $k, k'\in \{1,...,q_n\}$ with $k\neq k'$ and $y\in[0,1)$ we have
\begin{equation*}
d(R^k(y), R^{k'}(y)) \geq d(k\alpha, k'\alpha) \geq d(|k-k'|\alpha,0) \geq d(q_{n-1} \alpha,0) =\delta_n \geq \frac{1}{q_n+q_{n-1}}.
\end{equation*}
Hence there is at most one such point in the interval $\left[ 0,\frac{1}{q_n+q_{n-1}}\right)$.
\end{proof}

Denote 
\begin{equation*}
f_n(x)=
\begin{cases}
f(x) & \text{if } x\in \left[ \frac{1}{q_n+q_{n-1}}, 1 \right),\\
0 & \text{if } x\in \left[ 0,\frac{1}{q_n+q_{n-1}}\right).
\end{cases}
\end{equation*}

\begin{lemma}\label{errlem}
For $N\geq 1$ we have
\begin{equation*}
S_N^{b_n + 1}(f)(x)\leq S_N(f_n)(x)\leq S_N^{b_n + 1}(f)(x) + 3N \;\;\;\forall x\in [0,1).
\end{equation*}
\end{lemma}
\begin{proof}
By Lemma \ref{countlem} there are at most $b_n + 1$ points of the form $R^j(x)$ with $0\leq j\leq N-1$ that are in $\left[ 0,\frac{1}{q_n+q_{n-1}}\right)$. 
Let $k'\leq b_n + 1$ be such that $R^{j_{k'}^N(x)}(x)\in \left[ 0,\frac{1}{q_n+q_{n-1}}\right)$ but $R^{j_{k' + 1}^N(x)}(x)\not\in \left[ 0,\frac{1}{q_n+q_{n-1}}\right)$ then\footnote{Besides the points $R^{j_{1}^N(x)}(x),\ldots,R^{j_{k'}^N(x)}(x)$ there might additionally be one orbit point in $\left[0,\frac{1}{q_n+q_{n-1}}\right)$, namely in the case $R^{j_{N}^N(x)}(x)=0$. However, even in this case, the calculation below remains true, since $f(0)=f_n(0)=0$.}
\begin{align*}
S_N^{b_n+1}(f)(x)=\sum_{i=b_n + 2}^N f\big(R^{j_i^N(x)}(x)\big) \leq \sum_{i=k'+1}^N f\big(R^{j_i^N(x)}(x)\big) = S_N(f_n)(x).
\end{align*}
On the other hand
\begin{align*}
S_N(f_n)(x)&=\sum_{i=k'+1}^N f\big(R^{j_i^N(x)}(x)\big) 
\leq \sum_{i=b_n + 2}^N f\big(R^{j_i^N(x)}(x)\big) + (b_n + 1 -k') f((q_n+q_{n-1})^{-1})\\
& \leq \sum_{i=b_n + 2}^N f\big(R^{j_i^N(x)}(x)\big) + (b_n + 1) (q_n+q_{n-1}) \leq S_N^{b_n + 1}(f)(x) + 3 N.
\end{align*}
\end{proof}

Using the Denjoy-Koksma inequality, we can immediately deduce a useful estimate for $S^1_{q_n}(f)$.
\begin{proposition}\label{qn1/xprop}
For $x\in [0,1)$ and $n\geq 1$ it holds that
\begin{equation*}
\Big|S_{q_n}^1(f)(x) - q_n \log(q_n)\Big| \leq 7q_n.
\end{equation*}
\end{proposition}
\begin{proof}
By Denjoy-Koksma inequality we have
\begin{equation*}
\Big|S_{q_n}(f_n)(x)- q_n \int_0^1 f_n(y) \d y\Big|\leq \Var(f_n),
\end{equation*}
and the claim follows from Lemma \ref{errlem}.
\end{proof}

We will use a more general form of this computation (see \cite[Lemma 4.4]{Berk_2020}).

\begin{lemma}\label{qjlem}
For $n\geq 1$ it holds that
\begin{equation}\label{qjest}
2 q_{n+1} \geq \sum_{j=1}^{n} a_j q_j.
\end{equation}
\end{lemma}
\begin{proof}
Recall that $q_{n+1}= a_n q_n + q_{n-1}$. For $n=1$, \eqref{qjest} becomes $2q_2 \geq a_1 q_1$, which, considering that $q_1=1$ and $q_2=a_1$, clearly holds. Now let $n\geq 2$ and assume \eqref{qjest} holds for all $k<n$, then
\begin{equation*}
\sum_{j=1}^{n} a_j q_j \leq a_n q_{n} + a_{n-1} q_{n-1} + 2q_{n-1} \leq (a_n+1)q_n + 2q_{n-1} \leq 2 q_{n+1}.
\end{equation*}
\end{proof}

\begin{lemma}\label{gendklem}
For each $n\geq 1, i\in [0,n-1]$ we have
\begin{equation*}
|S_{q_{n-i}}(f_n)(x) - q_{n-i} \log(q_{n-i})| \leq \min\left(2q_n, \frac{1}{x_{\min}^{q_{n-i}}} \right) + 3q_{n-i} \;\;\;\forall x\in [0,1).
\end{equation*}
\end{lemma}
\begin{proof}
Denote $h =1_{[\frac{1}{2q_{n-i}},1)} f_n$. Since any two points in $\{x,...,R^{q_{n-i} -1}(x)\}$ are at distance at least $\delta_{n-i}=|\alpha q_{n-i-1}|\geq \frac{1}{2 q_{n-i}}$, we have
\begin{equation*}
\# \left(\{x,...,R^{q_{n-i}-1}(x)\}\cap \left(0,\frac{1}{2 q_{n-i}}\right)\right) \leq 1,
\end{equation*}
hence
\begin{equation*}
|S_{q_{n-i}}(f_n)(x) - S_{q_{n-i}}(h)(x)|\leq f_n(x_{\min}^{q_{n-i}}) \leq \min\left(2q_n, \frac{1}{x_{\min}^{q_{n-i}}} \right).
\end{equation*}
Applying Denjoy-Koksma to $h$ yields 
\begin{equation*}
|S_{q_{n-i}}(h)(x) - q_{n-i} \log(q_{n-i})| \leq \Var(h) + \log(2) q_{n-i} \leq 3q_{n-i},
\end{equation*}
and the claim follows.
\end{proof}

\begin{lemma}\label{strong1/xlembn}
If $N\in[q_n, q_{n+1})$ then
\begin{equation*}
\left|S_N(f_n)(x)- \sum_{j=1}^n b_j q_j \log (q_j) \right| \leq 16 N + 2 \sum_{1\leq j \leq n-1, b_j \not=0} a_j q_j \log(b_j).
\end{equation*}
\end{lemma}
\begin{proof}
For $j=1,...,n-1$ and $s=0,...,b_j-1$ denote $x(j,s)=R^{\sum_{i=j+1}^n b_i q_i + s q_j}(x)$, and $x(n,s)=R^{sq_n}(x)$ for $s=0,...,b_n-1$. 
Then 
\begin{equation*}
S_N(f_n)(x)=\sum_{j=1}^n \sum_{s=0}^{b_j-1} S_{q_j}(f_n)(x(j,s)),
\end{equation*}
and by Lemma \ref{gendklem} we have
\begin{align}
\left|S_N(f_n)(x)- \sum_{j=1}^n b_j q_j \log (q_j) \right| &\leq 3N+ \sum_{j=1}^n \sum_{s=0}^{b_j-1} \min\left(2q_n, \frac{1}{(x(j,s))^{q_j}_{\min}} \right)\notag\\
&\leq 3N + 2b_n q_n + \sum_{j=1}^{n-1} \sum_{s=0}^{b_j-1} \min\left(2q_n, \frac{1}{(x(j,s))^{q_j}_{\min}} \right)\notag\\
&\leq 5N + \sum_{j=1}^{n-1} \sum_{s=0}^{b_j-1} \min\left(2q_n, \frac{1}{(x(j,s))^{q_j}_{\min}} \right).\label{strong1/xlembnerr}
\end{align}
For $j=1,...,n$ and $s=0,...,b_j-1$ denote $\tilde{x}(j,s)=(x(j,s))^{q_j}_{\min}$. For $j=1,...,n$ with $b_j\not=0$ it holds that
\begin{align*}
\{R^{\sum_{l=j+1}^n b_l q_l}(x), R^{\sum_{
l=j+1}^n b_l q_l + 1}(x),...,R^{\sum_{l=j}^n b_l q_l - 1}(x)\} &=  \bigcup_{s=0}^{b_j-1} \{x(j,s),...,R^{q_j-1}(x(j,s))\} \\
& \text{ and the union is disjoint}.
\end{align*}
The set on the left is a finite orbit of length $b_j q_j < q_{j+1}$, therefore all of the points in the set on the right are at distance at least $\frac{1}{2 q_{j+1}}$, in particular it holds that
\begin{equation}\label{tildexdist}
|\tilde{x}(j,s) - \tilde{x}(j,s')| \geq \frac{1}{2 q_{j+1}},
\end{equation}
for any $s, s'\in \{0,\ldots, b_j-1\}$, $s\neq s'$.

Similarly, for $k=1,...,n-1$, we have
\begin{align*}
\{R^{\sum_{j=k+1}^n b_j q_j}(x), R^{\sum_{j=k+1}^n b_j q_j + 1},...,R^{N-1}(x)\} &= \bigcup_{j=1}^k \bigcup_{s=0}^{b_j-1} \{x(j,s),...,R^{q_j-1}(x(j,s))\} \\
& \text{ and the union is disjoint}.
\end{align*}
The set on the left is a finite orbit of length $\sum_{j=1}^k b_j q_j < q_{k+1}$ and therefore, by Lemma \ref{countlem}, contains at most one point in $\left[ 0,\frac{1}{2q_{k+1}}\right)$. Letting $(j^*_{k},s^*_{k})$ with $1\leq j^*_k \leq k$ and $0\leq s^*_k \leq b_{j^*_k}-1$ be the index with $\tilde{x}(j^*_{k},s^*_{k})=\min_{j=1,...,k, \; s=0,...,b_j-1} \tilde{x}(j,s)$, it follows that\footnote{It is possible that also $\tilde{x}(j^*_k, s^*_k) \geq \frac{1}{2q_{k+1}}$, we do not make any claims of these points.}
\begin{equation}\label{tildexbig}
\tilde{x}(j,s) \geq \frac{1}{2q_{k+1}} \;\;\; j=1,...,k, \; s=0,...,b_j-1, \; (j,s)\not =(j^*_k,s^*_k).
\end{equation}
Let 
\begin{align*}
&\Delta^*=\{(j^*_1,s^*_1),...,(j^*_{n-1},s^*_{n-1})\} \; \text{ and} \\
&\Delta=\{(j,s)\;|\; 1\leq j\leq n-1, \; 0\leq s\leq b_j-1, \; (j,s)\not \in \Delta^*\}.
\end{align*}
We will split up the sum in \eqref{strong1/xlembnerr} as
\begin{equation}\label{strong1/xlembnerrsplit}
\begin{aligned}
\sum_{j=1}^{n-1} \sum_{s=0}^{b_j-1} \min\left(2q_n, \frac{1}{(x(j,s))^{q_j}_{\min}} \right) &= \sum_{(j,s)\in \Delta} \min\left(2q_n, \frac{1}{(x(j,s))^{q_j}_{\min}} \right) \\
&+ \sum_{(j,s)\in \Delta^*} \min\left(2q_n, \frac{1}{(x(j,s))^{q_j}_{\min}} \right).
\end{aligned}
\end{equation}

For the first sum, notice that \eqref{tildexbig} implies
\begin{equation}\label{tildexleast}
\tilde{x}(j,s) \geq \frac{1}{2q_{j+1}} \;\;\;\forall (j,s)\in \Delta.
\end{equation}
Using \eqref{tildexdist}, \eqref{tildexleast} and Lemma \ref{qjlem}, we obtain
\begin{equation}\label{Deltaineq}
\begin{aligned}
\sum_{(j,s)\in \Delta}& \min\left(2q_n, \frac{1}{(x(j,s))^{q_j}_{\min}} \right)\\ 
&\leq \sum_{j=1}^{n-1} \sum_{s=0}^{b_j-1} \frac{1}{\frac{1}{2q_{j+1}} + \frac{s}{2 q_{j+1}}}\\
&\leq 2 \sum_{1\leq j \leq n-1, b_j \not=0} q_{j+1} \left(\log(b_j)+1\right)\\
&\leq 2 \sum_{1\leq j \leq n-1, b_j \not=0} a_j q_j \left(\log(b_j)+1\right) + 2 \sum_{1\leq j \leq n-1, b_j \not=0} q_{j-1} \left(\log(b_j)+1\right)\\
&\leq 2 \sum_{1\leq j \leq n-1, b_j \not=0} a_j q_j \log(b_j) + 3 \sum_{j=1}^{n-1} a_j q_{j}\\
&\leq 2 \sum_{1\leq j \leq n-1, b_j \not=0} a_j q_j \log(b_j) + 3 N.
\end{aligned}
\end{equation}

Note that the indices $(j^*_k,s^*_k)$ may coincide for different $k$, in fact if $j^*_k=k-i$ for some $i\geq 1$, then $(j^*_k,s^*_k)=(j^*_{k-l},s^*_{k-l})$ for $l=1,...,i$. In other words, by \eqref{tildexbig}, for each $1\leq k\leq n-2$, it holds that
\begin{equation*}
\tilde{x}(j^*_k,s^*_k) < \frac{1}{2q_{k+2}} \implies (j^*_k,s^*_k) = (j^*_{k+1},s^*_{k+1}).
\end{equation*}
Using Lemma \ref{qjlem}, it follows that 
\begin{equation}\label{Deltastarineq}
\sum_{(j,s) \in \Delta^*} \min\left(2q_n, \frac{1}{\tilde{x}(j,s)}\right) \leq 2q_n + 2 \sum_{k=1}^{n-2} q_{k+2} \leq 4q_n+2\sum_{k=1}^{n-1} q_k \leq 8 N.
\end{equation}

From \eqref{strong1/xlembnerr}, \eqref{strong1/xlembnerrsplit}, \eqref{Deltaineq} and \eqref{Deltastarineq} it follows that
\begin{equation*}
\left|S_N(f_n)(x)- \sum_{j=1}^n b_j q_j \log (q_j) \right| \leq 16 N + 2 \sum_{1\leq j \leq n-1, b_j \not=0} a_j q_j \log(b_j).
\end{equation*}
\end{proof}

Recall that the first goal will be to show that
\begin{equation*}
\lim_{N\to\infty}\frac{S_N^{b_n + 1}(f)(x)}{N\log(N)} = 1 \quad \text{uniformly in } x\in [0,1),
\end{equation*}
if $\alpha$ is of Roth type. Considering Lemmas \ref{errlem} and \ref{strong1/xlembn}, this will follow, once we show
\begin{equation}\label{err1/xbnest}
\sum_{1\leq j \leq n-1, b_j \not=0} a_j q_j \log(b_j) = o (N \log(N)),
\end{equation}
and
\begin{equation}\label{eq: NlogN/bjqjlogqj}
\lim_{N\to \infty} \frac{N \log(N)}{\sum_{j=1}^n b_j q_j \log(q_j)} = 1.
\end{equation}
First we will show \eqref{err1/xbnest}.

\begin{lemma}\label{err1/xbnlem}
If $\lim_{n\to \infty} \frac{\log(q_{n+1})}{\log(q_n)}=1$, then 
\begin{equation*}
\sum_{j=1}^{n-1} a_j q_j \log(a_j) = o (q_n \log(q_n)).
\end{equation*}
\end{lemma}
\begin{proof}
Let $\epsilon>0$, $n_0\geq 1$ be so big that $q_{m+1}<q_{m}^{1+\epsilon}$ for all $m\geq n_0$ and thus in particular $a_m<q_m^{\epsilon}$. Moreover, let $n_1\geq n_0 + 1$ be so big that $q_{n_1} > \frac{q_{n_0}}{\epsilon}$. Then, using Lemma \ref{qjlem}, for $n\geq n_1$ it holds that
\begin{align*}
\sum_{j=1}^{n-1} a_j q_j \log(a_j) &= \sum_{j=1}^{n_0 - 1} a_j q_j \log(a_j) + \sum_{j=n_0}^{n - 1} a_j q_j \log(a_j)\\
&\leq \log(q_n) \sum_{j=1}^{n_0 - 1} a_j q_j + \epsilon \log(q_n) \sum_{j=n_0}^{n - 1} a_j q_j\\
& \leq 2 \log(q_n) (q_{n_0} + \epsilon q_n)\\
&\leq 2\epsilon q_n \log(q_n).
\end{align*} 
\end{proof}

This immediately implies \eqref{err1/xbnest}.
The next lemmas will be a preparation for Proposition \ref{logprop} which will give \eqref{eq: NlogN/bjqjlogqj}.

\begin{lemma}\label{logmlem1}
It holds that
\begin{equation*}
\sum_{j=1}^n b_j q_j \log (q_j) \leq N \log (N) \leq \sum_{j=1}^n b_j q_j \log (q_j)+ N\log \left(\sum_{j=1}^n b_j\right).
\end{equation*}
\end{lemma}
\begin{proof}
Since $N=\sum_{j=1}^n b_j q_j$ it is clear that $\sum_{j=1}^n b_j q_j \log (q_j) \leq N \log (N)$. Now note that the function $z\mapsto z \log(z)$ is convex on $(0,\infty)$, hence
\begin{align*}
\sum_{j=1}^n \frac{b_j}{\sum_{i=1}^n b_i} q_j \log (q_j) &\geq \left(\sum_{j=1}^n \frac{b_j}{\sum_{i=1}^n b_i} q_j\right) \log \left( \sum_{j=1}^n \frac{b_j}{\sum_{i=1}^n b_i} q_j \right)
\geq \frac{N}{\sum_{i=1}^n b_i} \log \left(\frac{N}{\sum_{i=1}^n b_i}\right).
\end{align*}
Multiplying this equation by $\sum_{i=1}^n b_i$ yields
\begin{equation*}
\sum_{j=1}^n b_j q_j \log (q_j) \geq N \log(N) - N \log \left(\sum_{i=1}^n b_i\right).
\end{equation*}
\end{proof}

\begin{lemma}\label{logmlem2}
If $\lim_{n\to \infty} \frac{\log(q_{n+1})}{\log(q_n)}=1$, then 
\begin{equation*}
\log \left(\sum_{j=1}^n b_j\right) = o( \log(N)).
\end{equation*}
\end{lemma}
\begin{proof}
Let $\epsilon\in (0,1)$ and $M'\geq 1$ be so big that $q_{m+1} < q_m^{1+\epsilon}$, which implies $b_m<q_m^{\epsilon}$, for $m>M'$. Let $M>M'$ be so big that $\sum_{j=1}^{M'} a_j < q_M^{\epsilon}$ and\footnote{If $\alpha=\frac{1+\sqrt{5}}{2}$ then $q_n=\ffloor{\alpha^n}{\sqrt{5}}$. For the golden ratio all the CFE coefficients are $1$, so for general $\alpha$ we have $q_n\geq \ffloor{\alpha^n}{\sqrt{5}}$.} $q_m>m^{\frac{1}{\epsilon}}$ for $m\geq M$. Then, for $N\geq q_M$ it holds that
\begin{align*}
\log \left(\sum_{j=1}^n b_j\right)& \leq \log \left(\sum_{j=M'}^n q_j^{\epsilon}\right) +\log(q_M^{\epsilon}) \\
&\leq \log(q_n^{\epsilon}) + \log(n) + \log(q_M^{\epsilon}) \leq 4\epsilon \log(N).
\end{align*}
\end{proof}

\begin{proposition}\label{logprop}
It holds that 
\begin{equation}\label{logclaim}
\lim_{N\to \infty} \frac{N \log(N)}{\sum_{j=1}^n b_j q_j \log(q_j)} = 1
\end{equation}
if and only if $\lim_{n\to \infty} \frac{\log(q_{n+1})}{\log(q_n)}=1$.
\end{proposition}
\begin{proof}
(i) If $\lim_{n\to \infty} \frac{\log(q_{n+1})}{\log(q_n)}=1$, then \eqref{logclaim} holds by Lemmas \ref{logmlem1} and \ref{logmlem2}.

(ii) On the other hand, if $\liminf_{n\to \infty} \frac{\log(q_{n+1})}{\log(q_n)}>1$ then there is an $\epsilon>0$ and a subsequence $\s{n}{l}$ such that $q_{n_l+1} > q_{n_l}^{1+\epsilon}$, which implies $a_{n_l}>q_{n_l}^{\epsilon/2}$, for $l$ sufficiently large. For $N_l=a_{n_l} q_{n_l} \in [q_{n_l},q_{n_l+1}-1]$ and $l$ sufficiently large it holds that\footnote{In this case $a_{n_l}=b_{n_l}$.}
\begin{equation*}
N_l \log(N_l) > \Big(1+\frac{\epsilon}{2}\Big) a_{n_l} q_{n_l} \log(q_{n_l}),
\end{equation*}
contradicting \eqref{logclaim}.
\end{proof}

\begin{proposition}\label{strong1/xprop}
If $\lim_{n\rightarrow \infty} \frac{\log(q_{n+1})}{\log(q_n)}=1$, then it holds that
\begin{equation*}
\lim_{N\rightarrow \infty} \frac{S_N^{b_n+1}(f)(x)}{N \log (N)} = 1 \;\;\; \text{ uniformly in } x\in [0,1).
\end{equation*}
\end{proposition}
\begin{proof}
This is an immediate consequence of Lemmas \ref{errlem} and \ref{strong1/xlembn}, \eqref{err1/xbnest} and Proposition \ref{logprop}.
\end{proof}

Now we will show that, assuming $\lim_{n\rightarrow \infty} \frac{\log(q_{n+1})}{\log(q_n)}=1$, weak laws can be obtained even without trimming. As an intermediate step, we first prove weak laws under light trimming.

\begin{lemma}\label{weak1/x}
If $\lim_{n\rightarrow\infty} \frac{\log(q_{n+1})}{\log(q_n)}=1$ then 
\begin{equation}\label{weak1/xrotcon}
\lambda\left(\left|\frac{S_N^1(f)}{N\log(N)}-1\right|>\epsilon\right) \to 0 \;\;\;\as{n} \quad  \forall \epsilon>0.
\end{equation}
\end{lemma}
\begin{proof}
Let $\epsilon\in\left(0,\frac{1}{100}\right)$ and $M_0$ be big enough such that $\frac{\log(q_{n+1})}{\log(q_n)}<1+\frac{\epsilon^2}{10}$ for all $n\geq M_0$. For $n\geq M_0$ denote $\epsilon_n=10\frac{1}{\epsilon}\big(\frac{\log(q_{n+1})}{\log(q_n)}-1\big) < \epsilon$, and let
\begin{align*} 
B_n & =\left\{x\in [0,1) \;\left|\; d(x+j\alpha,\Z)\leq \frac{\epsilon_n}{q_n} \; \text{ for some } j\in \{0,...,q_n-1\} \right. \right\}\\
&= \bigcup_{j=0}^{q_n-1} \left[\Big(-j\alpha-\frac{\epsilon_n}{q_n}\Big)\mod 1,\Big(-j\alpha+\frac{\epsilon_n}{q_n}\Big)\mod 1\right],
\end{align*}
note that $\lambda(B_n)\leq 2\epsilon_n$. Note furthermore that $a_n\leq \frac{q_{n+1}}{q_n} = q_n^{-1+ \frac{\log(q_{n+1})}{\log(q_n)}}$, taking logarithms we obtain 
\begin{equation}\label{epsnineq}
\log(a_n) \leq \frac{\epsilon \epsilon_n}{10} \log(q_n) \leq \frac{\epsilon \epsilon_n}{2} \log(N).
\end{equation}

We will show that for big enough $n$ it holds that
\begin{equation}\label{mncond}
S_N^1(f)(x)-S_N^{b_n+1}(f)(x) \leq \epsilon N\log (N) + N \;\;\; \forall N\in [q_n,q_{n+1}-1],\,\, x\not \in B_n.
\end{equation}
Considering Proposition \ref{strong1/xprop}, \eqref{weak1/xrotcon} follows from \eqref{mncond}. Note that, if $a_n < 10$, or equivalently $q_{n+1} < 10 q_n + q_{n-1}$, then \eqref{mncond} holds trivially, even without any assumption on $x$, because
\begin{equation*}
S_N^1(f)(x)-S_N^{b_n+1}(f)(x) \leq \sum_{i=2}^{b_n} f(R^{j_i^N(x)}(x)) \leq 10 q_{n+1} \leq 110 N.
\end{equation*}
So for the rest of the proof we can focus on the case $a_n\geq 10$.

We will only show \eqref{mncond}, for $n$ with $\alpha-\frac{p_n}{q_n}>0$. The conclusion for all other $n$ follows by considering $1-\alpha$ instead\footnote{Let $n$ be such that $\alpha-\frac{p_n}{q_n}<0$ and denote $\tilde{\alpha}=1-\alpha$. Then the CFE approximants of $\tilde{\alpha}$ are given by $\frac{\tilde{p}_{\tilde{n}}}{\tilde{q}_{\tilde{n}}}=\frac{q_n-p_n}{q_n}$ (the index $\tilde{n}$ is either $n-1$ or $n+1$, depending on whether $\alpha<\frac{1}{2}$) and we have $\tilde{\alpha}-\frac{\tilde{p}_{\tilde{n}}}{\tilde{q}_{\tilde{n}}}>0$. Denote by $\tilde{S}_N$ the ergodic sum w.r.t.\ the rotation $\tilde{R}(x)=x+\tilde{\alpha}$, and note that for $N \in [q_n,q_{n+1} - 1] =[\tilde{q}_{\tilde{n}}, \tilde{q}_{\tilde{n} + 1} -1]$ it holds that $S_N^k = \tilde{S}_N^k \circ R^N$ for all $k\leq N$. Since $\tilde{\alpha}-\frac{\tilde{p}_{\tilde{n}}}{\tilde{q}_{\tilde{n}}}>0$, \eqref{mncond} yields
\begin{equation*}
\tilde{S}_N^1(f)(x)-\tilde{S}_N^{\tilde{b}_{\tilde{n}} +1}(f)(x) \leq \epsilon N\log (N)  \;\;\; \forall N\in [\tilde{q}_{\tilde{n}},\tilde{q}_{\tilde{n}+1} - 1], x\not \in \tilde{B}_{\tilde{n}},
\end{equation*}
where $\tilde{b}_{\tilde{n}}$ and $\tilde{B}_{\tilde{n}}$ are defined in the same way as $b_n$ and $B_n$, only using $\tilde{q}_{\tilde{n}}$ resp.\ $\tilde{q}_{\tilde{n}+1}$ instead of $q_n$ resp.\ $q_{n+1}$. Since $\tilde{q}_{\tilde{n}}=q_n$ and $\tilde{q}_{\tilde{n}+1}=q_{n+1}$ it holds that $\tilde{b}_{\tilde{n}}=b_n$ and $\tilde{B}_{\tilde{n}}=B_n$, it follows that 
\begin{equation*}
S_N^1(f)(x)-S_N^{b_n + 1}(f)(x) \leq \epsilon N\log (N)  \;\;\; \forall N\in [q_n, q_{n+1}-1], x\not \in R^N(B_n).
\end{equation*}
We will use an argument similar to this one several times in the sequel.
}
.

We first claim that, for $x\not \in B_n$ and big enough $n$, whenever $q_n\leq N <q_n + \epsilon_n q_{n+1}$, it holds that 
\begin{equation}\label{weak1/xj}
j_1^{q_n}(x)=j_1^N(x).
\end{equation}
Indeed, since $b_n\leq 1 + \frac{\epsilon_n q_{n+1}}{q_{n}}$, keeping in mind that $\epsilon_n<\frac{1}{100}$, and identifying $\T$ with $\left[-\frac{1}{2}, \frac{1}{2}\right)$, it holds that
\begin{equation*}
R^{j_{q_n}^{q_n}(x)}(x) \leq - \delta_n + \frac{\epsilon_n}{q_n} \leq - \frac{1}{2q_n} + \frac{1}{100 q_n} \leq - \frac{1}{q_{n+1}} - \frac{1}{100 q_n} \leq - b_n \delta_{n+1},
\end{equation*}
where we recall that $\delta_n = d(q_{n-1} \alpha, \Z) \in \left(\frac{1}{q_n + q_{n-1}}, \frac{1}{q_n}\right)$. Now \eqref{weak1/xj} follows by applying Lemma \ref{j1lem}. Therefore, in this case, $x^N_{\min} = x^{q_n}_{\min} \geq \frac{\epsilon_n}{q_n}$ and using \eqref{epsnineq} we obtain
\begin{equation*}
S_N^1(f)(x)-S_N^{b_n + 1}(f)(x) \leq b_n \frac{q_n}{\epsilon_n} \leq \frac{N}{\epsilon_n} \leq \epsilon N\log(N).
\end{equation*}

On the other hand, if $N\geq q_n + \epsilon_n q_{n+1}$, then it is not necessarily true that $j_1^{q_n}(x)=j_1^N(x)$. Nevertheless, using \eqref{epsnineq}, we can estimate
\begin{align*}
S_N^1(f)(x)&-S_N^{b_n+1}(f)(x) =\sum_{i=2}^{b_n+1} f(R^{j_i^N(x)}(x)) \leq \sum_{i=1}^{b_n} \frac{q_{n+1}}{i} \\
&\leq q_{n+1} (1 + \log(b_n)) \leq \frac{1}{\epsilon_n} N (1+\log(b_n))\leq \epsilon N\log(N) +N,
\end{align*}
and we have shown \eqref{mncond} in both cases.
\end{proof}

\begin{lemma}\label{weak1/xnotlem}
Suppose that, for some $\gamma\in (0,1]$ and $n\in \N$, it holds that $q_{n+1}>q_n^{1+\gamma}$ and $q_n^{\gamma}>\gamma^{-1} 1000$. Then there are sets $A,B\subset [0,1)$ with $\lambda(A)=\lambda(B)=\frac{\gamma}{1000}$ such that, whenever $k\leq q_{n+1} q_n^{-\left(1+\frac{\gamma}{2}\right)}$,
\begin{equation*}
S_{N}^{k}(f)(x)>\frac{\gamma}{4} q_{n+1} \log(q_n), \; \text{ while } \; S_{N}^{k}(f)(y)<\frac{\gamma}{100} q_{n+1} \log(q_n) \;\;\;\forall x\in A, y\in B,
\end{equation*}
where $N=\fceil{\gamma q_{n+1}}{250}$. 
\end{lemma}
\begin{proof}
We only treat the case for $\alpha-\frac{p_n}{q_n}>0$, the other case follows by considering $1-\alpha$ instead. Let
\begin{align*}
A&=\left\{x\in [0,1) \;\left|\; R^{j_{q_n}^{q_n}(x)}(x)\in \left(-\frac{\gamma}{1000 q_n},0\right)\right. \right\}
=\bigcup_{j=0}^{q_n-1} \left(-\frac{\gamma}{1000 q_n},0\right) - j\alpha,
\end{align*}
and
\begin{align*}
B&=\left\{x\in [0,1) \;\left|\; R^{j_{q_n}^{q_n}(x)}(x) \in \left(-\left(\frac{1}{4}+\frac{\gamma}{1000}\right) \frac{1}{q_n},-\frac{1}{4 q_n}\right)\right. \right\}\\
&=\bigcup_{j=0}^{q_n-1} \left(-\left(\frac{1}{4}+\frac{\gamma}{1000}\right) \frac{1}{q_n},-\frac{1}{4 q_n}\right) - j\alpha.
\end{align*}
Clearly $\lambda(A)=\lambda(B)=\frac{\gamma}{1000}$.

For $x\in B$ we have
\begin{equation*}
R^{j_{q_n}^{q_n}(x)}(x) \leq -\frac{1}{4 q_n} \leq -\fceil{\gamma q_{n+1}}{250 q_n} \frac{1}{q_{n+1}} \leq -b_n \delta_{n+1},
\end{equation*}
and Lemma \ref{j1lem} yields $j_1^{q_n}(x)=j_1^N(x)$. Moreover, since different points in $\{x,...,R^{q_n - 1}(x)\}$ are at distance at least $\frac{1}{2q_n}$, it holds that $R^{j_1^{q_n}(x)}(x) \geq \frac{1}{8q_n}$, therefore $S_{N}^{k}(f)(x)=S_{N}^{k}(f')(x)$, where
\begin{equation*}
f'(z)=\begin{cases}
0 & \text{if } z\in [0,\frac{1}{8q_n}),\\
f(z) & \text{otherwise}.
\end{cases}
\end{equation*}
Using the Denjoy-Koksma inequality, we obtain
\begin{equation*}
S_{N}^{k}(f)(x) \leq S_{(b_n + 1) q_n}(f')(x) \leq 2 (b_n+1) q_n \log(q_n) + 8 (b_n+1) q_n \leq \frac{\gamma}{100} q_{n+1} \log(q_n).
\end{equation*}

On the other hand, since $\alpha-\frac{p_n}{q_n}>0$, applying $R^{q_n}$ will move any point to the right by a distance of $\delta_{n+1} = | q_n \alpha| > \frac{1}{2q_{n+1}}$. Therefore, if $x\in A$ then it holds that 
\begin{equation*}
R^{j_{q_n}^{q_n}(x) + \fceil{b_n}{2} q_n}(x) \geq -\frac{\gamma}{1000 q_n} + \fceil{b_n}{2} \delta_{n+1} \geq -\frac{\gamma}{1000 q_n} + \frac{\gamma q_{n+1}}{500 q_n} \frac{1}{2q_{n+1}} \geq 0,
\end{equation*}
where we identify the circle with $\left[-\frac{1}{2}, \frac{1}{2}\right)$. Since now the points 
\begin{equation*}
R^{j_{q_n}^{q_n}(x)}(x),R^{j_{q_n}^{q_n}(x) +q_n}(x),...,R^{j_{q_n}^{q_n}(x) + \fceil{b_n}{2} q_n}(x)
\end{equation*}
are at distance $\delta_{n+1} < \frac{1}{q_{n+1}}$ and cross from negative to positive, it follows that there is an $s^*\in\left\{1,...,\fceil{b_n}{2} \right\}$ such that $R^{j_{q_n}^{q_n}(x)+s^* q_n}(x)\in\left(0, \frac{1}{q_{n+1}}\right)$. There could be two such points in the interval, in that case choose the left one\footnote{For the calculation below it is not important which we choose.} so that $j_{q_n}^{q_n}(x)+s^* q_n = j_1^N(x)$. The points
\begin{equation*}
\{R^{j_{q_n}^{q_n}(x)+(s^*+k) q_n}(x),..., R^{j_{q_n}^{q_n}(x)+(b_n - 1) q_n}(x)\} 
\end{equation*}
have $k$ other points to their left, therefore $S^k_N(f)(x)$ can be estimated from below by summing $f$ only over those points. More precisely
\begin{align*}
S^k_N(f)(x)&>\sum_{j=s^*+k}^{b_n-1} f(R^{j_{q_n}^{q_n}(x)+j q_n}(x)) >\sum_{j=k}^{\ffloor{b_n-1}{2}} f\left(\frac{j+1}{q_{n+1}}\right)\\
& = q_{n+1} \sum_{j=k}^{\ffloor{b_n-1}{2}} \frac{1}{j+1} \geq q_{n+1} \left(\log \left(\ffloor{b_n-1}{2}\right) - \log(k) \right)\\
& \geq \frac{1}{2} q_{n+1} \left( \log\left(\frac{q_{n+1}}{2q_n}\right)-\log\Big(q_{n+1} q_n^{-1-\frac{\gamma}{2}}\Big)\right) \geq \frac{\gamma}{4} q_{n+1} \log(q_n).
\end{align*}
\end{proof}

Now we are in a position to prove Theorem \ref{weak1/xthm}.

\begin{proof}[Proof of Theorem \ref{weak1/xthm}]
First assume that $\lim_{n\to\infty} \frac{\log(q_{n+1})}{\log(q_n)} = 1$, then, by Lemma~\ref{weak1/x}
\begin{equation}\label{weak1/xsm1con}
\lambda\left(\left|\frac{S_N^1(f)(x)}{N\log(N)}-1\right|>\epsilon\right) \to 0 \;\;\;\as{n} \quad \forall \epsilon>0.
\end{equation}
For $N\geq 1$ let $M_N=S_N(f)-S_N^1(f)=\max_{i=0,...,N-1} f\circ R^i$, then, for $\epsilon>0$, we have
\begin{align*}
\lambda&\left(\left|\frac{S_N(f)(x)}{N\log(N)}-1\right|>\epsilon\right) \leq\lambda\left(M_N(x) > \frac{\epsilon}{2} N \log(N)\right) + \lambda\left(\left|\frac{S_N^1(f)(x)}{N\log(N)}-1\right|>\frac{\epsilon}{2}\right)\\
&\leq \lambda\left(x\;\left|\; \exists i\in \{0,\ldots, N-1\}\text{ s.t.\ } R^i(x)\in \left( 0,  \left(\frac{\epsilon}{2} N \log(N)\right)^{-1}\right) \right)\right. \\
&\qquad+ \lambda\left(\left|\frac{S_N^1(f)(x)}{N\log(N)}-1\right|>\frac{\epsilon}{2}\right)\\
&\leq \frac{2}{\epsilon \log(N)} + \lambda\left(\left|\frac{S_N^1(f)(x)}{N\log(N)}-1\right|>\frac{\epsilon}{2}\right).
\end{align*}
In light of \eqref{weak1/xsm1con}, it follows that
\begin{equation*}
\lambda\left(\left|\frac{S_N(f)(x)}{N\log(N)}-1\right|>\epsilon\right) \to 0 \;\;\;\as{n} \quad \forall \epsilon>0.
\end{equation*}

On the other hand, now suppose $\limsup_{n\to\infty} \frac{\log(q_{n+1})}{\log(q_n)} > 1$. Taking a subsequence $\s{n}{l}$ we may assume that there is a $\gamma\in (0,1)$ such that $q_{n_l+1}>q_{n_l}^{1+\gamma}$ and $q_{n_l}^{\gamma}>\gamma^{-1} 1000$. For a contradiction assume that there are $d_N>0$ such that
\begin{equation*}
\lambda\left(\left|\frac{S_N(f)(x)}{d_N}-1\right|>\epsilon\right) \to 0 \;\;\;\as{N}\quad \forall \epsilon>0.
\end{equation*}
For every $l\geq 1$ let $N_l=\fceil{\gamma q_{n_l+1}}{250}$, by Lemma \ref{weak1/xnotlem} there are sets $A_l,B_l$ with $\lambda(A_l)=\lambda(B_l)=\frac{\gamma}{1000}$ such that
\begin{equation}\label{weak1/xthmeq1}
S_{N_l}(f)(x)>\frac{\gamma}{4} q_{n_l+1} \log(q_{n_l})\;\;\;\forall x\in A_l,
\end{equation}
while
\begin{equation}\label{weak1/xthmeq2}
S_{N_l}(f)(y)<\frac{\gamma}{100} q_{n_l+1} \log(q_{n_l}) \;\;\;\forall y\in B_l.
\end{equation}
Due to \eqref{weak1/xthmeq1}, we must have $d_{N_l} \geq \frac{\gamma}{8} q_{n_l+1} \log(q_{n_l})$, on the other hand \eqref{weak1/xthmeq2} implies $d_{N_l} \leq \frac{\gamma}{50} q_{n_l+1} \log(q_{n_l})$. This is a contradiction.
\end{proof}

\begin{proof}[Proof of Remark \ref{seq1/xrem}]
Let $l(N)=o(N)$ be a sequence of natural numbers and $k(N)\leq l(N)$. For $\epsilon>0$ denote 
\begin{equation*}
u(\epsilon)=\min(N\geq 1\;|\; l(m)<\epsilon m, \;\forall m\geq N).
\end{equation*}
Let $\alpha$ be such that,
\begin{equation*}
q_{n+1}>q_n^2 u \left( \frac{1}{q_n^2} \right) \;\;\; \text{ for all $n$ sufficiently large},
\end{equation*}
noting that there is a $G_{\delta}$ dense set of $\alpha$ satisfying this condition. Moreover, this choice implies that
\begin{equation*}
k(N) \leq l(N) \leq q_{n+1} q_n^{-\frac{3}{2}},
\end{equation*}
whenever $N=\fceil{q_{n+1}}{250}$ and $\sqrt{q_n}> 250$. The claim follows from Lemma \ref{weak1/xnotlem} with $\gamma=1$.
\end{proof}

We will now focus on the claims about strong convergence, i.e.\ Theorems \ref{strong1/xthm}, \ref{nonconat0thm1/x}, and \ref{noncothm1/x}.

\begin{proof}[Proof of Theorem \ref{strong1/xthm}]
Let $\kappa=\frac{1}{8}$ and $\alpha$ be such that $q_{n+1} < q_n \log(q_n) \log^2(\log(q_n))$ and 
\begin{equation}\label{strong1/xthmfkalphacond}
\sum_{n\geq 1, q_{n+1} \geq q_n \log^{1-\frac{\kappa}{2}}(q_n)} \frac{1}{\log^{1-\kappa}(q_n)} <\infty,
\end{equation}
by\footnote{In \cite{fkmultimix} it is shown that the stronger condition 
\begin{equation*}
\sum_{n\geq 1, q_{n+1} \geq q_n \log^{1-\kappa}(q_n)} \frac{1}{\log^{1-\kappa}(q_n)} <\infty,
\end{equation*}
holds for almost all $\alpha$.}
\cite{fkmultimix}, this condition is satisfied for almost all $\alpha$. Denote 
\begin{equation*}
\mathcal{N}=\{n\geq 1 \;|\; q_{n+1} < q_n \log^{1-\frac{\kappa}{2}}(q_n)\}.
\end{equation*}

Clearly, $\alpha$ fulfils the Roth type condition and thus, Proposition \ref{strong1/xprop} shows that 
\begin{equation*}
\lim_{N\rightarrow \infty} \frac{S_N^{b_n +1}(f)(x)}{N \log (N)} = 1 \;\;\; \text{ uniformly in } x\in [0,1).
\end{equation*}
The claim \eqref{strong1/xthmcon} follows once we show that, for almost every $x$, it holds that
\begin{equation}\label{strong1/xwant}
S_N^1(f)(x) - S_N^{b_n + 1}(f)(x) = o(N \log(N)).
\end{equation}
We estimate the left side in two different ways;
\begin{itemize}
\item since any two distinct points in $\{x,...,R^{N-1}(x)\}$ are at distance at least $\frac{1}{2q_{n+1}}$, we have 
\begin{equation}\label{strong1/xproofineq1}
S_N^1(f)(x) - S_N^{b_n + 1}(f)(x) \leq \sum_{j=1}^{b_n} \frac{2 q_{n+1}}{j} \leq 10 q_{n+1} \max(1, \log(b_n)), 
\end{equation}
\item on the other hand, denoting $x^N_{\min}=R^{j_1^N(x)}(x)$, we can trivially estimate
\begin{equation}\label{strong1/xproofineq2}
S_N^1(f)(x) - S_N^{b_n + 1}(f)(x) \leq b_n \frac{1}{x^N_{\min}}.
\end{equation}
\end{itemize}

Since $b_n < \log(q_n) \log^2(\log(q_n))$, we have;
\begin{enumerate}
\item[(A)] if $n\in \mathcal{N}$ sufficiently large, then by \eqref{strong1/xproofineq1} it holds that
\begin{align*}
S_N^1&(f)(x) - S_N^{b_n + 1}(f)(x) \leq 10 q_{n+1} \max(1,\log(b_n)) 
\leq 20 q_n \log^{1-\frac{\kappa}{2}}(q_n) \log_2(q_n)=o(N \log(N)),
\end{align*}
\item[(B)] if $N > q_n \log^{\frac{\kappa}{2}}(q_n)$, then by \eqref{strong1/xproofineq1} for $n$ sufficiently large it holds that
\begin{align*}
S_N^1&(f)(x) - S_N^{b_n + 1}(f)(x) \leq 10 q_{n+1} \max(1,\log(b_n)) \\
&\leq 20 q_n \log(q_n) \log^3(\log(q_n)) \leq 100 N \log^{1-\frac{\kappa}{4}}(N) = o(N\log(N)),
\end{align*}
\item[(C)] if $x^N_{\min} \geq \frac{1}{q_n \log^{1-\kappa}(q_n)}$, then by \eqref{strong1/xproofineq2} for $n$ sufficiently large it holds that
\begin{equation*}
S_N^1(f)(x) - S_N^{b_n + 1}(f)(x) \leq q_n \log^{1-\kappa}(q_n) \log^2(\log(q_n)) = o(N\log(N)).
\end{equation*}
\end{enumerate}

In the next steps we will show that, indeed, for almost all $x$ and all but finitely many $n\not\in\mathcal{N}$, if $N\geq q_n\log^{\frac{\kappa}{2}}(q_n)$ we have $x_{\min}^N\geq \frac{1}{q_n\log^{1-\kappa}(q_n)}$ and we have covered all cases.

For each $n\not \in \mathcal{N}$ with $\alpha-\frac{p_n}{q_n}>0$ let $\epsilon_n= \log^{-(1-\kappa)}(q_n)$ and 
\begin{align*} 
B_n & =\left\{x\in [0,1) \;\left|\; d(x+j\alpha,\Z)\leq \frac{\epsilon_n}{q_n} \; \text{ for some } j\in \{0,...,q_n-1\} \right. \right\}
= \bigcup_{j=0}^{q_n-1} \left[\frac{-\epsilon_n}{q_n},\frac{\epsilon_n}{q_n}\right] - j \alpha.
\end{align*}
For $n\not \in \mathcal{N}$ with $\alpha-\frac{p_n}{q_n}>0$, $x\not \in B_n$, and $N \leq q_n \log^{\frac{\kappa}{2}}(q_n)$ we have
\begin{equation*}
R^{j_{q_n}^{q_n}(x)}(x) <  -q_n^{-1} \log^{-(1-\kappa)}(q_n) \leq - b_n q_{n+1}^{-1},
\end{equation*}
and by Lemma \ref{j1lem} it follows that $j_1^{q_n}(x)=j_1^N(x)$. In this case
\begin{equation*}
x^N_{\min} \geq q_n^{-1} \log^{-(1-\kappa)}(q_n).
\end{equation*}

Similarly, for $n\notin\mathcal{N}$, $\alpha-\frac{p_n}{q_n}<0$ and
\begin{equation*}
B_n'=\bigcup_{j=0}^{q_n-1} \left[0,\frac{2\epsilon_n}{q_n}\right] - j\alpha,
\end{equation*}
it also holds that 
\begin{equation*}
x^N_{\min} \geq q_n^{-1} \log^{-(1-\kappa)}(q_n) \quad \forall x\not\in B_n'.
\end{equation*}
Indeed, if $x\notin B_n'$, $\alpha-\frac{p_n}{q_n}<0$, $n\notin \mathcal{N}$ and $N\leq q_n \log^{\frac{\kappa}{2}}(q_n)$, then
\begin{equation*}
j_1^N(x)= \begin{cases}
j_1^{q_n}(x) + b_n q_n & \text{if } j_1^{q_n}(x) + b_n q_n \leq N,\\
j_1^{q_n}(x) + (b_n -1) q_n & \text{otherwise},
\end{cases}
\end{equation*}
therefore, under the just stated conditions, it holds that
\begin{equation*}
x^N_{\min} \geq 2 q_n^{-1} \log^{-(1-\kappa)}(q_n) - b_n q_{n+1}^{-1} \geq q_n^{-1} \log^{-(1-\kappa)}(q_n).
\end{equation*}

Since $\lambda( B_n\cup B_n')= 3\log^{-(1-\kappa)}(q_n)$, the assumption \eqref{strong1/xthmfkalphacond} shows that $\sum_{n\not \in \mathcal{N}} \lambda(B_n \cup B_n') <\infty$, and by Borel-Cantelli almost every $x$ is only in finitely many $B_n$ or $B_n'$ with $n\not\in \mathcal{N}$. By (A)-(C) it follows that \eqref{strong1/xwant} holds for almost every $x$.

For the "furthermore" part of the statement, note that if $\alpha$ is of bounded type then $\mathcal{N}=\N \setminus K$ for a finite set $K\subset \N$. Hence, case (C), the only case where we don't have a uniform statement, does not have to be considered.
\end{proof}

Note that, if $\alpha$ is not of bounded type, the proof uses a Borel-Cantelli argument, hence convergence typically is not uniform. For example we shall show that, for almost all $\alpha$,
\begin{equation*}
\frac{S_N^1(f)(0)}{N \log(N)}
\end{equation*}
does not converge.

\begin{lemma}\label{pointoscbiglem}
If $q_{n+1} \in (q_n \log(q_n) \log_3(q_n)\, ,\, q_n\log^2(q_n))$, $q_n>10^6$, $\alpha-\frac{p_n}{q_n}>0$ and $x^{q_n}_{\min} < \frac{1}{q_n \log (q_n) \log_3(q_n)}$, then it holds that
\begin{equation*}
S_N^1(f)(x)-S_N^{b_n+1}(f)(x) \geq \frac{1}{2} N\log(N),
\end{equation*}
where $N= \lceil \frac{q_{n + 1} \log_5(q_n)}{\log(q_{n})} \rceil $.
\end{lemma}
\begin{proof}
Denote $\epsilon=\frac{1}{q_n \log(q_n) \log_3(q_n)}$, for $n$ sufficiently large, it holds that
\begin{align*}
S_N^{1}(f)(x)-S_N^{b_n + 1}(f)(x) &\geq \sum_{i=1}^{b_n} \frac{1}{\epsilon + \frac{i}{q_{n+1}}}\geq \sum_{i=1+\lceil\epsilon q_{n+1}\rceil}^{b_n + \lceil\epsilon q_{n+1}\rceil} \frac{q_{n+1}}{i}\\
&\geq \frac{1}{2} q_{n+1} \log \left( \frac{b_n}{\epsilon q_{n+1}} \right) \geq \frac{1}{2} N\frac{\log(q_n)}{\log_5(q_n)} \log_4(q_n)
\geq \frac{1}{2} N \log (N).
\end{align*}
\end{proof}

\begin{remark}
The "$\log_3$" in the previous lemma might seem a bit odd, and in fact, the lemma works also if we replace $\log_3$ by any function $\psi$ with $\psi(n)\nearrow \infty$. The reason for our choice will be apparent from the proofs of Theorems \ref{nonconat0thm1/x} and \ref{noncothm1/x}.
\end{remark}

\begin{lemma}\label{evenlem}
For almost every $\alpha$ there are infinitely many $n\geq 1$ such that 
$$q_{n+1}> q_n \log(q_n) \log_3(q_n)\quad\text{ and }\alpha-\frac{p_n}{q_n}>0.$$
\end{lemma}
\begin{proof}
To avoid confusion, in this proof we shall write $q_n(\alpha)$ resp.\ $a_n(\alpha)$ instead of $q_n$ resp.\ $a_n$.

We have $\sum_{n\geq 1} \frac{1}{n \log(n) \log_3(n)} = \infty$ and hence by Khinchine's Theorem, for almost every $\alpha$, there are infinitely many $n$ with 
\begin{equation*}
q_{n+1}(\alpha)> q_n(\alpha) \log(q_n(\alpha)) \log_3(q_n(\alpha)).
\end{equation*}
Equivalently,\footnote{This is not quite equivalent since $q_{n+1}=a_n q_n + q_{n-1}$, but almost. Using Khinchine also yields $q_{n+1}(\alpha)> 2q_n(\alpha) \log(q_n(\alpha)) \log_3(q_n(\alpha))$ for almost all $\alpha$, which implies $a_n(\alpha) >  \log(q_n(\alpha))\log_3(q_n(\alpha))$.}
\begin{equation*}
a_n(\alpha) >  \log(q_n(\alpha))\log_3(q_n(\alpha)).
\end{equation*}

Expanding $\alpha$ as
\begin{equation}\label{cfedef2}
\alpha=\cfrac{1}{a_1(\alpha) + \cfrac{1}{a_2(\alpha) + \cfrac{1}{...}}},
\end{equation}
we can easily see that $\alpha-\frac{p_n(\alpha)}{q_n(\alpha)}>0$ if and only if $n$ is even. Denote
\begin{align*}
&A_{even}=\left\{\alpha\in (0,1)\setminus \Q \;\left|
\begin{aligned}
& \text{there are infinitely many even } n \text{ with }\\
&  a_n(\alpha) >  \log(q_n(\alpha))\log_3(q_n(\alpha))
\end{aligned}
\right\}\right. ,\\
&A_{odd}=\left\{\alpha\in (0,1)\setminus \Q \;\left|
\begin{aligned}
& \text{there are infinitely many odd } n \text{ with }\\
&  a_n(\alpha) >  \log(q_n(\alpha))\log_3(q_n(\alpha))
\end{aligned}
\right\}\right. .
\end{align*}

Let $T(x)=\ffloor{1}{x}$ be the Gauss-map. From \eqref{cfedef2} it follows that
\begin{equation*}
a_n(T(\alpha))=a_{n+1}(\alpha) \; \text{ and therefore } \; q_n(T(\alpha)) \leq q_{n+1}(\alpha),
\end{equation*} 
hence $A_{odd} \subset T^{-1} (A_{even})$ and $A_{even} \subset T^{-1} (A_{odd})$. It follows that $A_{even} \subset T^{-2}(A_{even})$. Since $T$ preserves the Gauss-measure $\mu$, which is equivalent to Lebesgue, we have $A_{even} = T^{-2}(A_{even})$ (mod $\mu$). And since $T^2$ is ergodic\footnote{A fortiori $T$ is mixing.} $\mu(A_{even})\in \{0,1\}$. From
\begin{equation*}
A_{even} \subset T^{-1} (A_{odd}) \subset T^{-2} (A_{even}) = A_{even} \text{ (mod $\mu$)},
\end{equation*}
it follows that $A_{even}=T^{-1} (A_{odd})$ (mod $\mu$), hence $\mu(A_{odd})=\mu(A_{even})\in \{0,1\}$, and by equivalence to Lebesgue also $\lambda(A_{odd})=\lambda(A_{even})\in \{0,1\}$. At the same time, by Khinchine it holds that $\lambda(A_{even} \cup A_{odd})=1$, it follows that $\lambda(A_{odd})=\lambda(A_{even})=1$ as desired.
\end{proof}

\begin{proof}[Proof of Theorem \ref{nonconat0thm1/x}] 
Proposition \ref{qn1/xprop} ascertains that $|S_{q_n}^1(f)(x)-q_n \log (q_n)|\leq 7q_n$ and the "$\liminf$" part of the claim holds. 

By Lemma \ref{evenlem}, for almost all $\alpha$, there are infinitely many $n$ with $q_{n+1}> q_n \log(q_n) \log_3(q_n)$ and $\alpha-\frac{p_n}{q_n}>0$. For all such $n$ and $x=\alpha$ we have 
\begin{equation*}
x^{q_n}_{\min}=R^{q_n}(0) < \frac{1}{q_{n+1}} < \frac{1}{q_n \log(q_n) \log_3(q_n)},
\end{equation*}
and by Lemma \ref{pointoscbiglem} there is an $N_n\in [q_n, q_{n+1}-1]$ such that 
\begin{equation}\label{noneq}
S_{N_n}^1(f)(x)-S_{N_n}^{b_n+1}(f)(x) \geq \frac{1}{2} N_n\log(N_n).
\end{equation}
The claim follows from \eqref{noneq} together with Proposition \ref{strong1/xprop}.
\end{proof}

Theorem \ref{noncothm1/x} can be proven using the same ideas. Here we will need a version of the second Borel-Cantelli Lemma, using only weak pairwise independence. The following has been proven in \cite{LAMPERTIBC}, see also \cite{PETROVBC} or \cite{CHANDRABC}.

\begin{lemma}[\cite{LAMPERTIBC}]\label{reclem}
Let $(\Omega, \P)$ be a probability space and $\s{A}{n}$ be a sequence of events with
\begin{equation*}
\sum_{n=1}^{\infty} \P(A_n) = \infty
\end{equation*}
and there is a $K>0$ such that
\begin{equation*}
\P(A_i \cap A_j) \leq K \P(A_i) \P(A_j) \;\;\;\forall i<j.
\end{equation*}
Then it holds that
\begin{equation*}%\label{recpropclaim}
\P(\omega\in \Omega \;|\; \omega \in A_n \text{ for infinitely many } n) > 0.
\end{equation*}
\end{lemma}

\begin{proof}[Proof of Theorem \ref{noncothm1/x}]
Let $\alpha$ be such that $q_{n+1}\in (q_n \log(q_n) \log_3(q_n), q_n\log^2(q_n))$ 
for all $n$. The claim about the weak law follows from Theorem \ref{weak1/xthm}. We will now show that the strong law does not hold. For simplicity, we only consider $n$ for which $\alpha-\frac{p_n}{q_n}>0$, the other case follows by considering $1-\alpha$ instead.

(i) For the "$\liminf$" part of the claim, note that by Proposition \ref{qn1/xprop} we always have $|S_{q_n}^1(f)(x)-q_n \log (q_n)|\leq 7q_n$. Furthermore, since the points $\{x,...,R^{q_n -1}(x)\}$ are at distance at least $\delta_n > \frac{1}{2q_n}$, it holds that $R^{j_{i+1}^{q_n}(x)}(x) \geq \frac{i}{2q_n}$. It follows that
\begin{equation*}
|S_{q_n}^K(f)(x)-q_n \log (q_n)|\leq (7+2\log(K)) q_n \;\;\; \text{ if } q_n>K \geq 1.
\end{equation*}

(ii) Let $N_n= \lceil \frac{q_{n + 1} \log_5(q_n)}{\log(q_{n})} \rceil $ and $G$ be the set 
\begin{equation*}
G=\Big\{x \;|\; \limsup_{n\to\infty} \frac{S_{N_n}^{K}(f)(x)}{N_n \log(N_n)} >1\Big\}.
\end{equation*}
Note that \footnote{Since, for big enough $n$, the distance between any two points in $\{R(x),...,R^{N}(x)\}$ is at least $\frac{1}{2q_{n+1}}$, at most one such point can be in $[0,\frac{\log_5(q_n)}{N_n \log(N_n)})\subset [0, \frac{1}{2q_{n+1}})$.}
\begin{equation*}
|S_{N_n}^K(f)(x)-S_{N_n}^K(f)(R(x))| = o(N_n \log(N_n)) \;\;\; \text{ uniformly in } x,
\end{equation*}
therefore $G=R(G)$ is invariant. By ergodicity, the claim \eqref{noncoclaim1/x} will follow once we show $\lambda(G)>0$.
 
(iii) Let $\epsilon_n=\frac{1}{q_n \log(q_n) \log_3(q_n)}$ and
\begin{equation*}
A_n=\{x\;|\; R^{j_1^{q_n}(x)}(x)\in (0,\epsilon_n)\} = \bigcup_{l=0}^{q_n -1} (0,\epsilon_n) - l\alpha,
\end{equation*}
and this union is disjoint.\footnote{Since $\epsilon_n < \frac{1}{2 q_n} < \min_{j_1,j_2\in \{0,..., q_n - 1\}, j_1 \not = j_2} |j_1 \alpha - j_2 \alpha|$.} Note that $\frac{b_n}{\epsilon_n q_{n+1}} \geq \frac{\log_3(q_n) \log_5(q_n)}{2}$. Then for $x\in A_n$
\begin{align*}
S^K_{N_n}(f)(x)-S^{b_n + 1}_{N_n}(f)(x) 
&\geq \sum_{i=K}^{b_n} \frac{1}{\epsilon_n + \frac{i}{q_{n+1}}}\geq \sum_{i=K+\lceil\epsilon_n q_{n+1}\rceil}^{b_n +\lceil\epsilon_n q_{n+1}\rceil} \frac{q_{n+1}}{i}\\
& \geq (1000K)^{-1} q_{n+1} \log \left(\frac{b_n}{\epsilon_n q_{n+1}}\right) \geq \kappa N_n \log (N_n),
\end{align*}
for some $\kappa>0$ and $n$ sufficiently large. 

By Proposition \ref{strong1/xprop} we have
\begin{equation*}
S_N^{b_n +1}(f)(x)= N \log (N) (1+o(1))\;\;\; \text{ uniformly in } x\in [0,1),
\end{equation*}
and it follows that
\begin{equation*}
G\supset \{x \;|\; x\in A_n \text{ for infinitely many } n\}.
\end{equation*}
It remains to show
\begin{equation*}
\lambda(x \;|\; x\in A_n \text{ for infinitely many } n)>0.
\end{equation*}

(iv) We will verify the conditions of Lemma \ref{reclem} for $A_n$, i.e.\ we will show
\begin{equation}\label{optsum}
\sum_{n=1}^{\infty} \lambda(A_n) = \infty
\end{equation}
and 
\begin{equation}\label{optind}
\lambda(A_i \cap A_j) \leq 2 \lambda(A_i) \lambda(A_j) \;\;\;\forall i<j.
\end{equation}
Lemma \ref{reclem} will then imply $\lambda (x \;|\; x\in A_n \text{ for infinitely many } n)>0$ and the proof will be complete.

(v) First we verify \eqref{optsum}. Indeed, the assumption $q_{n+1}<q_n \log^2(q_n)$ implies\footnote{This can be shown by induction. Recall that, by convention $q_1=1$. Now assume $q_n < 16^n (n!)^4$ for some $n$, then it holds that
\begin{align*}
q_{n+1}< 16^n (n!)^4 \log^2(16^n (n!)^4) \leq 16^n (n!)^4 (\log(16) n + 4 n \log(n))^2 \leq 16^{n+1} ((n+1)!)^4.
\end{align*}
} that $q_n < 16^n (n!)^4$, in particular $\log(q_n) < 100 n \log(n)$ and $\log_3(q_n) < 10 \log_2(n)$. It follows that
\begin{align*}
\sum_{n=1}^{\infty} \lambda(A_n) & = \sum_{n=1}^{\infty} \epsilon_n q_n = \sum_{n=1}^{\infty} \frac{1}{\log(q_n) \log_3(q_n) }
> \frac{1}{1000} \sum_{n=1}^{\infty} \frac{1}{n \log(n) \log_2(n)} =\infty.
\end{align*}

Now we verify \eqref{optind}. Let $i<j$ and, for $l_1=0,...,q_i - 1$ and $l_2=0,...,q_j - 1$, denote
\begin{equation*}
I_{l_1} = (0, \epsilon_i) - l_1 \alpha \; \text{ and } \; \tilde{I}_{l_2} = (0, \epsilon_j)  - l_2 \alpha,
\end{equation*}
additionally
\begin{equation*}
\mathcal{N}_{l_1} =\#\{l\in \{0,...,q_j -1\} \;|\; \tilde{I}_l \cap I_{l_1} \not=\emptyset\}.
\end{equation*}
Since the $\tilde{I}_l$ are disjoint we have
\begin{align*}
\mathcal{N}_{l_1}& = \#\{l\in \{0,...,q_j -1\} \;|\; - l \alpha \in I_{l_1}\} = \#\{l\in \{0,...,q_j -1\} \;|\; l \alpha \in -I_{l_1}\} =S_{q_j} (1_{-I_{l_1}})(0).
\end{align*}
\begin{equation*}
\lambda(A_i \cap A_j) \leq S_{q_j}(1_{A_i})(0) \epsilon_j \leq (q_j q_i \epsilon_i + 2) \epsilon_j \leq 2 \lambda(A_i) \lambda (A_j),
\end{equation*}
showing \eqref{optind} and completing the proof.
\end{proof}

\subsection{The function $x^{-\beta}$, $\beta>1$}

In this section, we shall provide a proof of Theorem \ref{betathm}. Let $\beta>1$ and $f(x)=x^{-\beta}$. Clearly, the strong law implies the weak law, so (I) $\implies$ (II).
The main propositions in the section are the following: We show that the strong law holds (uniformly) assuming \eqref{betacond}, i.e.\ (III) $\implies$ (I), see Proposition \ref{betastrongprop}.
Then we introduce an a priori weaker condition than \eqref{betacond} - condition (D) given in Definition \ref{def: condD}. We show in 
Proposition \ref{prop:II to D} that (II) $\implies$ condition (D) and Lemma \ref{betacondeqlem} shows that condition (D) and \eqref{betacond} are equivalent under the condition that $k(N)$ is monotone implying (II) $\implies$ (III).

We start with some preparations to prove Proposition \ref{betastrongprop}. To simplify matters, we will introduce an auxiliary sequence $K(N)=o(N)$. Instead of studying the full sum $S_N(f)$ it will be convenient to study separately\footnote{It holds that $S_N(f)=S_{N'}(f) + S_{N''}(f) \circ R^{N''}$.} the sums $S_{N'}(f)$ and $S_{N''}(f)$ with
\begin{equation*}
N'=\sum_{j=n-K(N)}^n b_j q_j, \quad \text{ and } \quad N''=N-N'=\sum_{j=1}^{n-K(N)-1} b_j q_j \leq q_{n-K(N)}.
\end{equation*}
By carefully choosing $K(N)$, the second sum $S_{N''}(f)$ is dominated by $S_{N'}(f)$, hence it suffices to study the latter.

Let $\epsilon_{\beta}>0$ be so that $\left(1+\epsilon_{\beta}\right) \frac{\beta - 1}{\beta} < 1$, and
\begin{equation*}
K(N)= \min \left(\kappa \geq 1 \;\left|\; q_{n - \kappa} \leq N k(N)^{\left(1+\epsilon_{\beta}\right) \frac{1-\beta}{\beta}} \right) \right. .
\end{equation*}
With this choice it holds that
\begin{align*}
\sup_{x\in [0,1)} |S_{N''}^{1}(f)(x)|&\leq \sup_{x\in [0,1)} |S_{q_{n-K(N)}}^{1}(f)(x)|
\leq \sum_{i=1}^{q_{n-K(N)}} q_{n-K(N)}^{\beta} i^{-\beta}\\
&= O\big(q_{n-K(N)}^{\beta} \big) =O\big(N^{\beta} k(N)^{(1+\epsilon_{\beta}) (1-\beta)}\big)= o\big(N^{\beta} k(N)^{1-\beta}\big).
\end{align*}

Since
\begin{equation}\label{betastrongthmsplit}
S_{N'}^{k(N)}(f) \leq S_N^{k(N)}(f)\leq S_{N'}^{k(N)-1}(f) + S_{N''}^{1}(f)\circ R^{N'},
\end{equation}
the strong law \eqref{betastrongcon}, with $d_N=\frac{1}{\beta-1} N^{\beta} k(N)^{1-\beta}$, will follow once 
we show\footnote{Applying the same arguments with $k(N)-1$ instead of $k(N)$ also yields
\begin{equation*}
\lim_{N\rightarrow\infty} \frac{S^{k(N)-1}_{N'}(f)(x)}{\frac{1}{\beta-1} N^{\beta} k(N)^{-\beta+1}} = 1 \;\;\; \text{uniformly in } x\in [0,1).
\end{equation*}}
\begin{equation}\label{betastrongthmclaim2}
\lim_{N\rightarrow\infty} \frac{S^{k(N)}_{N'}(f)(x)}{N^{\beta} k(N)^{-\beta+1}} = \frac{1}{\beta-1} \;\;\; \text{uniformly in } x\in [0,1).
\end{equation}

First, we claim that  
\begin{equation}\label{kncond}
\frac{k(N)}{\sum_{n-K(N)\leq j\leq n} b_j} \to \infty.
\end{equation}
Note that, if\footnote{To be slightly more precise; if $K(N)=1$ along a subsequence, then \eqref{kncond} follows trivially from \eqref{betacond} along this subsequence.} $K(N)=1$, then this follows trivially from \eqref{betacond}. Therefore we can focus on the case $K(N)>1$.

By \eqref{betacond}, the condition \eqref{kncond} is equivalent to the a priori weaker condition
\begin{equation*}
\frac{k(N)}{\sum_{n-K(N)+1 \leq j\leq n} b_j} \to \infty.
\end{equation*} 

Note that, for every $1\leq J \leq n-1$, $N$ satisfies
\begin{equation*}
N\geq b_n q_n > b_n a_{n-1} q_{n-1} > ... > b_n \prod_{j=n-J}^{n-1} a_j q_{n-J}.
\end{equation*}
Therefore, by the definition of $K(N)$ (and since $K(N)>1$), we have  
\begin{equation*}
q_{n-K(N)+1} > N k(N)^{(1+\epsilon_{\beta})\frac{1-\beta}{\beta}} > b_n \prod_{j=n-K(N)+1}^{n-1} a_j q_{n-K(N)+1} k(N)^{(1+\epsilon_{\beta})\frac{1-\beta}{\beta}}.
\end{equation*}  

Moreover, for sufficiently large $n$ and $K$, it holds that  
\begin{equation*}
q_n > \rho^{K-2} q_{n-K}, \quad \text{where } \rho = \frac{1+\sqrt{5}}{2} \text{ (the golden ratio).}
\end{equation*}  
It follows that $K(N) = o(k(n))$, and consequently,  
\begin{equation*}
\sum_{n-K(N)+1\leq j\leq n} b_j \leq b_n \prod_{j=n-K(N)+1}^{n-1} a_j + K(N) \leq k(N)^{(1+\epsilon_{\beta})\frac{\beta-1}{\beta}} + K(N) = o(k(N))
\end{equation*}  
and hence \eqref{kncond} holds.

Let $z(N,x)=R^{j^{N'}_{k(N)+1}(x)}(x)$, recall that this means that $z(N,x)$ is the $k(N)+1$st point in $\{x,...,R^{N'-1}(x)\}$ that we encounter moving from left to right from $0$ (not counting $0$ itself, if $0\in \{x,...,R^{N'-1}(x)\}$), 
then we have
\begin{equation}\label{zdef}
S^{k(N)}_{N'}(f)(x)=f(z(N,x))+\sum_{y\in \{x,...,R^{N'-1}(x)\}, y > z(N,x)} f(y).
\end{equation}
We first try to locate $z(N,x)$. To that end first rewrite
\begin{equation*}
z(N,x)=\min(\zeta\in\{x,...,R^{N'-1}(x)\} \;|\; S_{N'}(1_{(0,\zeta]})(x)\geq k(N)+1).
\end{equation*}
For $\zeta\in (0,1)$ we have 
\begin{equation*}
|S_{N'}(1_{(0,\zeta]})(x)-N'\zeta| \leq 2 \sum_{j=n-K(N)}^n b_j.
\end{equation*}
Since
\begin{equation*}
S_{N'}(1_{(0,\zeta]})(x) \in [N'\zeta - 2 \sum_{j=n-K(N)}^n b_j\, ,\, N' \zeta + 2 \sum_{j=n-K(N)}^n b_j],
\end{equation*}
it follows that 
\begin{equation*}
z(N,x)\in \left[\frac{k(N) + 1-2 \sum_{j=n-K(N)}^n b_j}{N'}, \frac{k(N) + 1 +2 \sum_{j=n-K(N)}^n b_j}{N'}\right] =: I_N. 
\end{equation*}

Denote 
\begin{equation*}
f_N(x)=\begin{cases}
f(x) & \text{if } x\in \left[\frac{k(N)+1}{N'},1\right),\\
0 & \text{otherwise}.
\end{cases}
\end{equation*}

\begin{lemma}\label{errbetalem}
For $N\in [q_n,q_{n+1}-1]$ sufficiently large
\begin{equation*}
|S_{N'}^{k(N)}(f)(x)-S_{N'}(f_N)(x)|=o(N^{\beta} k(N)^{1-\beta}) \;\;\; \text{uniformly in } x\in [0,1).
\end{equation*}
\end{lemma}
\begin{proof}
For fixed $x\in [0,1)$ let $\tilde{f}_N$ be the function defined by
\begin{equation*}
\tilde{f}_N(y)=\begin{cases}
f(y) & \text{ if } y\in \left[z(N,x),1\right),\\
0 & \text{ otherwise},
\end{cases}
\end{equation*}
then, by \eqref{zdef}, we have $S_{N'}^{k(N)}(f)(x)=S_{N'}(\tilde{f}_N)(x)$. Furthermore, for $N$ sufficiently large
\begin{align*}
|S_{N'}(\tilde{f}_N)(x)-S_{N'}(f_N)(x)|
&\leq 2^{\beta} N^{\beta} k(N)^{-\beta} \#\{0\leq j\leq N'-1 \;|\; R^j(x)\in I_N\}\\
&\leq 2^{\beta} N^{\beta} k(N)^{-\beta} S_{N'}(1_{I_N})(x)\\
&\leq 1000^{\beta} N^{\beta} k(N)^{-\beta} \sum_{j=n-K(N)}^n b_j,
\end{align*}
and the claim follows since $\sum_{j=n-K(N)}^n b_j=o(k(N))$.
\end{proof}

\begin{lemma} 
It holds that
\begin{equation}\label{betaeq1}
\left| S_{N'}^{k(N)}(f)(x)- \frac{1}{\beta-1} N^{\beta} k(N)^{-\beta+1} \right| =o(N^{\beta} k(N)^{-\beta+1}) \;\;\; \text{ uniformly in } x\in [0,1).
\end{equation}
\end{lemma}
\begin{proof}
The Denjoy-Koksma inequality yields
\begin{equation*}
|S_{N'}(f_N)(x)-N'\int_0^1 f_N(y) \d y|\leq N^{\beta} k(N)^{-\beta} \sum_{j=n-K(N)}^n b_j. 
\end{equation*}
Since $\int_0^1 f_N(y) \d y = \frac{1}{\beta-1} (N^{\beta-1} k(N)^{-\beta+1}-1)$ and $\frac{N}{N'}\to1$, the claim follows via Lemma \eqref{errbetalem}. 
\end{proof}

Altogether we obtain the following.
\begin{proposition}\label{betastrongprop}
If \eqref{betacond} holds, then the strong law holds uniformly, i.e.
\begin{equation*}
\lim_{N\to \infty} \frac{S_N^{k(N)}(f)(x)}{d_N} = \frac{1}{\beta - 1} \;\;\; \text{ uniformly in } x\in [0,1).
\end{equation*}
\end{proposition}
\begin{proof}
Considering \eqref{betastrongthmsplit}, it is enough to show \eqref{betastrongthmclaim2}, which follows from \eqref{betaeq1}.
\end{proof}

To complete the proof of Theorem \ref{betathm}, it remains to show that if the weak law\footnote{A priori weak or strong laws could also be possible for a different normalisation $d_N>0$ with 
\begin{equation*}
\lim_{N\to\infty} \frac{(\beta-1) d_N}{N^{\beta} k(N)^{1-\beta}} \not = 1,
\end{equation*}
so we have to take this into account in the proof.} holds and $k(N)$ is monotone, then $k(N)$ satisfies condition \eqref{betacond}.

To this end, we show that if the weak law \eqref{betaweakcon} is satisfied, then $k(N)$ satisfies the following (a priori weaker) condition.

\begin{definition}\label{def: condD}
$k(N)$ is said to satisfy condition (D), if, for any subsequence $\s{N}{l}$ satisfying $\limsup_{l\to \infty} \frac{N_l}{q_{n+1}}<1$, it holds that
\begin{equation}\label{betacond'}
\frac{k(N_l)}{b_{n}} \to \infty.
\end{equation} 
\end{definition}

\begin{lemma}\label{betacondeqlem}
Let $k(N)$ be a monotone sequence satisfying condition (D), then \eqref{betacond} holds.
\end{lemma}

Since clearly \eqref{betacond} is the stronger condition, the lemma shows that if $k(N)$ is assumed to be monotone, then (D) is equivalent to \eqref{betacond}.

\begin{proof}
If $\alpha$ is of bounded type, then both \eqref{betacond} and \eqref{betacond'} are equivalent to $k(N)\rightarrow\infty$. So assume now that $\alpha$ is not of bounded type.

Let $\mathcal{N}:= \{n \;|\; a_n>2\}$ and consider the sequence
\begin{equation*}
\s{N}{l} = (q_n,2q_n,..., \fceil{a_n}{2} q_n \;|\; n\in \mathcal{N}),
\end{equation*}
clearly $\limsup_{l\to \infty} \frac{N_l}{q_{n+1}}\leq \frac{2}{3}$ and, per assumption, \eqref{betacond'} holds for the sequence $\s{N}{l}$.

In order to show \eqref{betacond}, let $C>0$, by \eqref{betacond'} there is an $L_0$ so big that $\frac{k(N_l)}{b_{n}}>2C$ for $l\geq L_0$. Furthermore, let $n_0$ be such that $N_{L_0} \in [q_{n_0}, q_{n_0 + 1} - 1]$. Since $\alpha$ is not of bounded type, there is an $n_1 > n_0$ such that\footnote{Implicitly $n_1\in \mathcal{N}$.} $a_{n_1} > \max_{j\leq n_0} a_j$. Now \eqref{betacond} will follow once we show that, for\footnote{Now $N$ is not necessarily a member of the sequence $\s{N}{l}$ any more.} $N> q_{n_1 + 1}$, it holds that
\begin{equation*}
\frac{k(N)}{\max (b_n , \max_{j\leq n-1} a_j)} >C.
\end{equation*}
Since $n\geq n_1 +1$ and $n_1\in \mathcal{N}$, the above is equivalent to
\begin{equation*}
\frac{k(N)}{\max (b_n , \max_{j\leq n-1, j\in \mathcal{N}} a_j)} >C.
\end{equation*}

For $j\in [n_0+1, n-1]\cap \mathcal{N}$ let $l_j\geq 1$ be such that $N_{l_j} = \fceil{a_n}{2} q_n$, and, if $n\in \mathcal{N}$, let $l_n\geq 1$ be such that $N_{l_n} =\min(b_n,\fceil{a_n}{2})$. Clearly ,
\begin{equation*}
\min_{j\in [n_0+1, n]} l_j \geq L_0  \quad \text{ and } \quad N \geq \max_{j\in [n_0+1, n]} N_{l_j}.
\end{equation*}

Since $k$ is monotone
\begin{itemize}
\item $k(N)\geq k(N_{l_{n_1}})  > C a_{n_1} > C \max_{j\leq n_0} a_j$,
\item for $j \in [n_0 + 1 , n-1]$ it holds that $k(N) \geq k(N_{l_j}) > C a_j$ since $N \geq \fceil{a_{n_1}}{2} q_{n_1}$,
\item $k(N) \geq k(N_{l_n}) > C b_n$ since $N \geq b_n q_n$.
\end{itemize}
\end{proof}

Assume now that there are $d_N>0$ such that 
\begin{equation*}
\lambda\left(\left|\frac{S_N^{k(N)}(f)}{d_N} - 1\right| > \epsilon\right) \to 0 \;\;\;\as{N},
\quad \forall\epsilon>0.
\end{equation*}
We will show that $k(N)$ satisfies condition (D). 

To this end, we will show that if $k(N)$ does not satisfy (D), then $S_N^{k(N)}(f)$ has "oscillations" of order $N^{\beta} k(N)^{1-\beta}$ (Lemma \ref{betaosclem}). To conclude, we have to ensure additionally that, along suitable subsequences $\s{N}{l}$, it holds that\footnote{This is not quite what we show, but morally speaking.} $d_{N_l} = O(N_l^{\beta} k(N_l)^{1-\beta})$.

\begin{lemma}\label{smkbetabiglem}
There is a constant $C>0$ only depending on $\beta$ such that, for $\epsilon\in (0,\frac{1}{100})$, $N \in [q_n,(1-\epsilon)q_{n+1}]$, $k\geq 0$ and $n$ big enough, we have
\begin{equation*}
\lambda\left(S_N^k(f) \leq C \epsilon^{-\beta} \min\left(N q_n^{\beta-1}, N^{\beta} k^{1-\beta}\right)\right) \geq \frac{\epsilon}{40}.
\end{equation*}
\end{lemma}
\begin{proof}
Assume $\alpha - \frac{p_n}{q_n} > 0$, the other case can be proven analogously. Let $A\subset [0,1)$ be the set
\begin{align*}
A&=\left\{x\in [0,1) \;\left|\; R^{j_1^{q_n}(x)}(x)\in \left[\frac{\epsilon}{20} \delta_n, \frac{\epsilon}{10} \delta_n\right] \right. \right\}
=\bigcup_{j=0}^{q_n-1} \left[\frac{\epsilon}{20} \delta_n, \frac{\epsilon}{10} \delta_n\right] - j \alpha,
\end{align*}
where $\delta_n$ is given by \eqref{DefDelta}, clearly $\lambda(A)\geq \frac{\epsilon}{40}$. We have $b_n \leq (1-\epsilon) \frac{q_{n+1}}{q_n}$. Hence, by Lemma \ref{rotdeltalem}, for $x\in A$ it holds that
\begin{equation*}
R^{j_{q_n}^{q_n}(x)}(x) \leq -(1-\frac{\epsilon}{10}) \delta_n \leq -(1-\frac{\epsilon}{10}) a_n \delta_{n+1} \leq - b_n \frac{1}{q_{n+1}},
\end{equation*}
and Lemma \ref{j1lem} ascertains that $j_1^N(x)=j_1^{q_n}(x)$. By summing separately over the first cluster of points, we obtain 
\begin{equation*}
\begin{aligned}
S_N^k(f)(x)
&\leq \sum_{j=k}^{b_n+1} \left(\frac{\epsilon}{40 q_n} + \frac{j}{2q_{n+1}}\right)^{-\beta} + b_n \sum_{i=\max\left(1,\ffloor{k}{b_n}\right)}^{q_n - 1} \left( \frac{i}{2 q_n} \right)^{-\beta}\\
&\leq 100^{\beta} \epsilon^{-\beta} \chi_{k\leq b_n+1} b_n q_n^{\beta} + \frac{2^{\beta}}{\beta-1} b_n q_n^{\beta} \min(1,b_n^{\beta-1} k^{1-\beta})\\
&\leq C \epsilon^{-\beta} \min(N q_n^{\beta-1}, N^{\beta} k^{1-\beta}),
\end{aligned}
\end{equation*}
for some constant $C>0$, where
\begin{equation*}
\chi_{k\leq b_n+1} = \begin{cases}
1 & \text{ if } k\leq b_n+1\\
0 & \text{ otherwise}.
\end{cases}
\end{equation*}
\end{proof}

\begin{lemma}\label{betadnlem}
If the weak law of large numbers as in \eqref{betaweakcon} holds and $\s{N}{l}$ is a subsequence fulfilling $\limsup_{l\to \infty} \frac{N_l}{q_{n+1}}<1$, then
\begin{equation}\label{dnO}
d_{N_l} = O(\min(N_l q_n^{\beta-1} , N_l^{\beta} k(N_l)^{1-\beta})).
\end{equation}
\end{lemma}
\begin{proof}
Let $\epsilon \in (0,100^{-1})$ and $L$ be such that $\frac{N_l}{q_{n+1}}<1-\epsilon$ for $l>L$. For such $l$, Lemma \ref{smkbetabiglem} yields
\begin{equation*}
\lambda(S_{N_l}^k \leq C \epsilon^{-\beta} \min(N_l q_n^{\beta-1}, N_l^{\beta} k(N_l)^{1-\beta})) \geq \frac{\epsilon}{40}.
\end{equation*}
This clearly implies 
\begin{equation*}
d_{N_l} = O(\min(N_l q_n^{\beta-1}, N_l^{\beta} k(N_l)^{1-\beta})).
\end{equation*}
\end{proof}

\begin{remark}
A posteriori, if $k(N)$ satisfies condition (D), then 
\begin{equation*}
N_l^{\beta} k(N_l)^{1-\beta} = o (N_l q_n^{\beta -1} ),
\end{equation*}
and \eqref{dnO} becomes simply
\begin{equation*}
d_{N_l}=O(N_l^{\beta} k(N_l)^{1-\beta}).
\end{equation*} 
\end{remark}

\begin{lemma}\label{floorlem}
For $\epsilon\in (0,1)$ and $w\geq 1$ it holds that
\begin{equation*}
\lfloor (1- \epsilon) w \rfloor - 1 \leq (1-\frac{\epsilon}{2}) \lfloor w \rfloor.
\end{equation*}
\end{lemma}
\begin{proof}
If $w\leq \frac{2}{\epsilon}$, then
\begin{equation*}
\frac{\lfloor (1- \epsilon) w \rfloor - 1}{\lfloor w \rfloor} \leq 1 - \frac{1}{w} \leq 1 - \frac{\epsilon }{2}.
\end{equation*}
On the other hand, if $w > \frac{2}{\epsilon}$
\begin{equation*}
\lfloor (1- \epsilon) w \rfloor = \lfloor w - \epsilon w \rfloor \leq \lfloor w \rfloor - \lfloor \epsilon w \rfloor,
\end{equation*}
and
\begin{equation*}
\lfloor \epsilon w \rfloor \geq \epsilon w - 1 \geq (1-\frac{1}{2}) \epsilon w \geq \frac{\epsilon}{2} \lfloor w \rfloor.
\end{equation*}
It follows that
\begin{equation*}
\lfloor (1- \epsilon) w \rfloor \leq (1-\frac{\epsilon}{2}) \lfloor w \rfloor.
\end{equation*}
\end{proof}

\begin{lemma}\label{betaosclem}
There is a constant $c>0$ such that, for $\epsilon\in (0,\frac{1}{100})$ and $N \in [q_n,(1-\epsilon)q_{n+1}]$, there are sets $A,B$ with $\lambda(A)=\lambda(B)\geq \frac{\epsilon^2}{1000}$ such that
\begin{equation}\label{rotbetanotweakclaim1}
S_{N}(f)(x)-S_{N}(f)(y)> c \epsilon N q_n^{\beta-1} \;\;\;\forall x\in A, y\in B.
\end{equation}
Furthermore, if $k= \hat{k} \frac{N}{q_n} \in \N$ then
\begin{equation}\label{rotbetanotweakclaim2}
S_{N}^k(f)(x)-S_{N}^k(f)(y)> c \epsilon \min\left(N q_n^{\beta - 1} ,\hat{k}^{-2} N^{\beta} k^{1-\beta} \right) \;\;\;\forall x\in A, y\in B.
\end{equation}
\end{lemma}
\begin{proof}
We focus on \eqref{rotbetanotweakclaim2}, \eqref{rotbetanotweakclaim1} can be shown using almost the same arguments. In the following we assume $\alpha-\frac{p_n}{q_n}>0$, the other case follows by considering $1-\alpha$ instead.

Let
\begin{align*}
A'&=\{x\in (0,1] \;|\; R^{j_1^{q_n}(x)}(x)\in (0,\frac{\epsilon}{10} \delta_n) \; \text{ and } \; j_1^{q_n}(x)\in \{0,..., q_n - q_{n-1} - 1\}\}\\
&=\bigcup_{j=0}^{q_n-q_{n-1} - 1} (0,\frac{\epsilon}{10} \delta_n) - j \alpha.
\end{align*}
Since $\delta_n\geq\frac{1}{2 q_n}$, it holds that\footnote{Strictly speaking this is only true if $(1-\frac{\epsilon}{2}) q_n> q_{n-1}$. If $(1-\frac{\epsilon}{2}) q_n\leq q_{n-1}$, then a similar proof with the set
\begin{align*}
A'&=\{x\in (0,1] \;|\; R^{j_1^{q_n}(x)}(x)\in (0,\frac{\epsilon}{20} \delta_{n-1})\}
=\bigcup_{j=0}^{q_n - 1} (0,\frac{\epsilon}{20} \delta_{n-1}) - j \alpha
\end{align*}
leads to the same conclusion.
} $\lambda(A') \geq \frac{\epsilon^2}{40}$. For $x\in A'$ we have
\begin{equation*}
R^{j_{q_n}^{q_n}(x)}(x)<R^{j_{q_n}^{q_n}(x)+q_n}(x)<...< R^{j_{q_n}^{q_n}(x) + (b_n-1) q_n} (x) < R^{j_1^{q_n}(x)}(x),
\end{equation*}
and no other points\footnote{The only candidates for points between $R^{j_{q_n}^{q_n}(x)}(x)$ and $R^{j_{1}^{q_n}(x)}(x)$ are
\begin{equation*}
R^{j_{q_n}^{q_n}(x)}(x)<R^{j_{q_n}^{q_n}(x)+q_n}(x)<...< R^{j_{q_n}^{q_n}(x) + b_n q_n} (x) < R^{j_1^{q_n}(x)}(x).
\end{equation*}
However the assumption $j_1^{q_n}(x) \leq q_n - q_{n-1} - 1$, or equivalently $j_{q_n}^{q_n} \geq q_{n-1}$, ensures that the point $R^{j_{q_n}^{q_n}(x) + b_n q_n} (x)$ is not present in $\{x,...,R^{N-1}(x)\}$.} of $\{x,...,R^{N-1}(x)\}$ are between $R^{j_{q_n}^{q_n}(x)}(x)$ and $R^{j_{1}^{q_n}(x)}(x)$. Using Lemma \ref{floorlem}, we have
\begin{equation*}
b_n - 1 \leq \ffloor{(1-\epsilon)q_{n+1}}{q_n} - 1 \leq \Big(1-\frac{\epsilon}{2}\Big) \ffloor{q_{n+1}}{q_n} = \Big(1-\frac{\epsilon}{2}\Big) a_n,
\end{equation*} 
in addition Lemma \ref{rotdeltalem} yields $a_n \delta_{n+1} \geq \delta_n$ and therefore
\begin{align*}
R^{j_{q_n}^{q_n}(x) + (b_n-1) q_n} (x) &\leq -\Big(1-\frac{\epsilon}{10}\Big) \delta_n + (b_n-1) \delta_{n+1}
 \leq -\Big(1-\frac{\epsilon}{10}\Big) \delta_n + \Big(1-\frac{\epsilon}{2}\Big) a_n \delta_{n+1}
\leq -\frac{\epsilon}{4} \delta_n.
\end{align*}
It follows that $j_1^N(x)=j_1^{q_n}(x)$ and $j_i^N(x)=j_i^N(x+\frac{\epsilon}{4}\delta_n)$ for all $i=1,...,N$. Writing 
\begin{equation*}
S_{N}^k(f)(x)-S_{N}^k(f)(x+\frac{\epsilon}{4}\delta_n) \geq \sum_{l=\lceil \hat{k} \rceil +1}^{q_n} \sum_{i=0}^{b_n-1} \left(f\left(R^{j_l^{q_n}(x)+i q_n}(x)\right) - f\left(R^{j_l^{q_n}(x)+i q_n}(x) +\frac{\epsilon}{4}\delta_n\right) \right),
\end{equation*}
we notice that all the summands are positive, hence we can bound below by taking only the terms where $l=\lceil \hat{k} \rceil +1$. Furthermore, all of the points $R^{j_{\lceil \hat{k}\rceil+1}^{q_n}(x)+i q_n}(x)$ are to the left of $R^{j_{\lceil \hat{k} \rceil +2}^{q_n}(x)}(x)$, therefore we again bound from below by replacing them all with this point. Therefore
\begin{equation}\label{betatranineq}
\begin{aligned}
S_{N}^k(f)(x)-S_{N}^k(f)(x+\frac{\epsilon}{4}\delta_n)
&\geq b_n \left(f\left(R^{j_{\lceil \hat{k} \rceil +2}^{q_n}(x)}(x)\right) - f\left(R^{j_{\lceil \hat{k}\rceil +2}^{q_n}(x)}(x) +\frac{\epsilon}{4}\delta_n\right) \right)\\
&\geq b_n \left( \left(\frac{\hat{k}+3}{q_n}\right)^{-\beta} -  \left(\frac{\hat{k}+3}{q_n}+\frac{\epsilon}{8q_n} \right)^{-\beta}\right)\\
&\geq c \epsilon \min(1,\hat{k}^{-\beta-1}) b_n  q_n^{\beta} \geq c \epsilon \min\left(N q_n^{\beta - 1} ,\hat{k}^{-2} N^{\beta} k^{1-\beta} \right) ,
\end{aligned}
\end{equation}
for a constant $c>0$.

We distinguish two cases:

(i) If there is an $s$ such that $\lambda(x\in A' \;|\; S_{N}^k(f)(x)=s )\geq \frac{1}{3} \lambda(A')$, we set
\begin{equation*}
A= \{x\in A' \;|\; S_{N}^k(f)(x)=s\} \; \text{ and } \; B=A + \frac{\epsilon}{4} \delta_n.
\end{equation*}
Then, for $x\in A,y\in B$, it holds that
\begin{equation*}
S_{N}^k(f)(x)-S_{N}^k(f)(y)=S_{N}^k(f)(y-\frac{\epsilon}{4} \delta_n)-S_{N}^k(f)(y) \overset{\eqref{betatranineq}}{>} c \epsilon \min\left(N q_n^{\beta - 1} ,\hat{k}^{-2} N^{\beta} k^{1-\beta} \right),
\end{equation*}
where in the first equality we used the fact that $S_N^k(f)$ is constant on $A$.

(ii) Otherwise, if there is no such $s$, let
\begin{equation*}
\mathcal{S} = \left\{s>0 \;\left|\; \lambda(x\in A' \;|\; S_{N}^k(f)(x)\leq s )\geq \frac{1}{3} \lambda(A')\right. \right\} \; \text{ and } \; s_0=\inf \mathcal{S}.
\end{equation*}
Then, for\footnote{Clearly $\mathcal{S}\not=\emptyset$ because $S_{N}^k(f)(x)<100^{\beta} \epsilon^{-\beta} N q_n^{\beta}$ whenever $x_{\min}^N>\frac{\epsilon}{10}\delta_n$.}
\begin{equation*}
A=\{x\in A' \;|\; S_{N}^k(f)(x)> s_0\} \; \text{ and } \; B=\{x\in A' \;|\; S_{N}^k(f)(x) \leq s_0\} + \frac{\epsilon}{4} \delta_n,
\end{equation*}
it holds that, for $x\in A,y\in B$, we have
\begin{equation*}
S_{N}^k(f)(x)-S_{N}^k(f)(y)>S_{N}^k(f)(y-\frac{\epsilon}{4} \delta_n)-S_{N}^k(f)(y) \overset{\eqref{betatranineq}}{>} c \epsilon \min\left(N q_n^{\beta - 1 },\hat{k}^{-2} N^{\beta} k^{1-\beta} \right),
\end{equation*}
where we used that $S_{N}^k(f)(x)>s_0$ by definition of $A$, and $S_{N}^k(f)(y-\frac{\epsilon}{4} \delta_n) \leq s_0$ by definition of $B$.

Furthermore, by continuity of $\lambda$, we have
\begin{equation*}
\lambda(B)=\lambda\left(\bigcap_{s\in \mathcal{S}} \{x\in A' \;|\; S_{N}^k(f)(x)\leq s\} \right) \geq \frac{1}{3} \lambda(A'). 
\end{equation*}
In order to estimate $\lambda(A)$, we distinguish two possibilities
\begin{itemize}
\item if $s_0\not \in \mathcal{S}$, then by definition it holds that
\begin{equation*}
\lambda(A)=\lambda(x\in A' \;|\; S_N^k(f)(x) > s_0) \geq \frac{2}{3} \lambda(A'),
\end{equation*}
\item if $s_0\in \mathcal{S}$, then
\begin{align*}
\lambda(A)&=\lambda(x\in A' \;|\; S_N^k(f)(x) > s_0)\\
&=\lambda\bigg( \bigcup_{s\not \in \mathcal{S}} \{x\in A' \;|\; S_N^k(f)(x) > s\} \bigg) - \lambda\left(x\in A' \;|\; S_N^k(f)(x)=s_0\right)\geq \frac{1}{3} \lambda(A').
\end{align*}
\end{itemize}
\end{proof}

\begin{proposition}\label{prop:II to D}
 If there are some $d_N>0$ such that 
\begin{equation*}
\lambda\left(\left|\frac{S_N^{k(N)}(f)}{d_N} - 1\right| > \epsilon\right) \to 0 \;\;\;\as{N},
\quad \forall\epsilon>0,
\end{equation*}
then $k(N)$ fulfils Property D. 
\end{proposition}
\begin{proof}
 We assume weak convergence as in \eqref{betaweakcon} holds and let $\s{N}{l}$ be a subsequence with $\lim_{l\to\infty} \frac{N_l}{q_{n+1}} < 1$, say $\lim_{l\to\infty} \frac{N_l}{q_{n+1}} = 1-\epsilon$ for some small $\epsilon>0$. By Lemma \ref{betadnlem}, we have 
\begin{equation}\label{betanotthmdm}
d_{N_l} = O(\min(N_l q_n^{\beta-1}, N_l^{\beta} k(N_l)^{1-\beta})).
\end{equation}
We claim that necessarily $\frac{k(N_l)}{b_n}\to\infty$, equivalently $\frac{k(N_l) q_n}{N_l} \to \infty$. If not, then, by Lemma\footnote{A priori it might also be possible that $k(N_l)=0$, in this case, use \eqref{rotbetanotweakclaim1}. Otherwise, we use \eqref{rotbetanotweakclaim2}.} \ref{betaosclem}, there are sets $A_{N_l},B_{N_l}$ with $\lambda(A_{N_l})=\lambda(B_{N_l})\geq \frac{\epsilon^2}{1000}$ with
\begin{equation*}
S_{N_l}^{k(N_l)}(f)(x)-S_{N_l}^{k(N_l)}(f)(y)> \tilde{c} \min\left(N_l q_{n_l}^{\beta - 1} ,N_l^{\beta} k(N_l)^{1-\beta} \right), \;\;\;\forall x\in A_{N_l}, y\in B_{N_l},
\end{equation*}
for a constant\footnote{The dependence on $\epsilon$ and "$\hat{k}$" as in Lemma \ref{betaosclem} can be absorbed into $\tilde{c}$. This is doesn't cause any problems since $\hat{k}=\frac{k(N_l) q_n}{N_l}$ is bounded by assumption.} $\tilde{c}>0$ which depends on $\epsilon$ but not on $l$. In light of \eqref{betanotthmdm}, this contradicts \eqref{betaweakcon}.
\end{proof}

Finally, we are in a position to give the full proof of Theorem \ref{betathm}.

\begin{proof}[Proof of Theorem \ref{betathm}]
Clearly, strong convergence implies weak convergence, therefore (I) $\implies$ (II). Additionally, if $k(N)$ satisfies condition \eqref{betacond}, then Proposition \ref{betastrongprop} shows that the strong law holds uniformly, so (III) $\implies$ (I). 
Moreover, by Proposition \ref{prop:II to D} we have that (II) $\implies$ condition (D) and according to Lemma \ref{betacondeqlem}, if $k(N)$ is monotone, condition (D) is equivalent to \eqref{betacond}.
\end{proof}

\providecommand{\bysame}{\leavevmode\hbox to3em{\hrulefill}\thinspace}
\providecommand{\MR}{\relax\ifhmode\unskip\space\fi MR }
\providecommand{\MRhref}[2]{%
  \href{http://www.ams.org/mathscinet-getitem?mr=#1}{#2}
}
\providecommand{\href}[2]{#2}


\begin{thebibliography}{Lam63}

\bibitem[Aar77]{Aaronson1977OnTE}
J.~Aaronson, \emph{On the ergodic theory of non-integrable functions and
  infinite measure spaces}, Israel J. Math. \textbf{27} (1977), no.~2,
  163--173.

\bibitem[Aar97]{aaronson1997introduction}
J.~Aaronson, \emph{An introduction to infinite ergodic theory}, Mathematical
  Surveys and Monographs, vol.~50, American Mathematical Society, Providence,
  RI, 1997.

\bibitem[AD01]{Adpsimix}
J.~Aaronson and M.~Denker, \emph{Local limit theorems for partial sums of
  stationary sequences generated by {G}ibbs-{M}arkov maps}, Stoch. Dyn.
  \textbf{1} (2001), no.~2, 193--237.

\bibitem[AN03]{AARONSONpsi}
J.~Aaronson and H.~Nakada, \emph{Trimmed sums for non-negative, mixing
  stationary processes}, Stochastic Process. Appl. \textbf{104} (2003), no.~2,
  173--192.

\bibitem[Aue25]{auer2025poissonlimittheoremssystems}
M.~Auer, \emph{Poisson limit theorems for systems with product structure},
  Discrete Contin. Dyn. Syst. \textbf{45} (2025), no.~5, 1454--1491.

\bibitem[BK21]{Berk_2020}
P.~Berk and A.~Kanigowski, \emph{Spectral disjointness of rescalings of some
  surface flows}, J. Lond. Math. Soc. (2) \textbf{103} (2021), no.~3, 901--942.

\bibitem[Bra05]{Bradleymixing}
R.~C. Bradley, \emph{Basic properties of strong mixing conditions. {A} survey
  and some open questions}, Probab. Surv. \textbf{2} (2005), 107--144, Update
  of, and a supplement to, the 1986 original.

\bibitem[Cha08]{CHANDRABC}
T.~K. Chandra, \emph{The {B}orel-{C}antelli lemma under dependence conditions},
  Statist. Probab. Lett. \textbf{78} (2008), no.~4, 390--395.

\bibitem[DF15]{dolgopyat2020limittheoremstoraltranslations}
D.~Dolgopyat and B.~Fayad, \emph{Limit theorems for toral translations},
  Hyperbolic dynamics, fluctuations and large deviations, Proc. Sympos. Pure
  Math., vol.~89, Amer. Math. Soc., Providence, RI, 2015, pp.~227--277.

\bibitem[DV86]{DIAMOND1986}
H.~G. Diamond and J.~D. Vaaler, \emph{Estimates for partial sums of continued
  fraction partial quotients}, Pacific J. Math. \textbf{122} (1986), no.~1,
  73--82.

\bibitem[Fel57]{feller1957introduction}
W.~Feller, \emph{An introduction to probability theory and its applications,
  volume 2}, An Introduction to Probability Theory and Its Applications, no. v.
  1-2, Wiley, 1957.

\bibitem[FK16]{fkmultimix}
B.~Fayad and A.~Kanigowski, \emph{Multiple mixing for a class of conservative
  surface flows}, Invent. Math. \textbf{203} (2016), no.~2, 555--614.

\bibitem[Hae93]{HaeuslernonstandardLIL}
E.~Haeusler, \emph{A nonstandard law of the iterated logarithm for trimmed
  sums}, Ann. Probab. \textbf{21} (1993), no.~2, 831--860.

\bibitem[Hay14]{HAYNES_2012}
A.~Haynes, \emph{Quantitative ergodic theorems for weakly integrable
  functions}, Ergodic Theory Dynam. Systems \textbf{34} (2014), no.~2,
  534--542.

\bibitem[HM87]{Haeusler_Mason_1987}
E.~Haeusler and D.~M. Mason, \emph{Laws of the iterated logarithm for sums of
  the middle portion of the sample}, Math. Proc. Cambridge Philos. Soc.
  \textbf{101} (1987), no.~2, 301--312.

\bibitem[Kes93]{Kesten_1993}
H.~Kesten, \emph{Convergence in distribution of lightly trimmed and untrimmed
  sums are equivalent}, Math. Proc. Cambridge Philos. Soc. \textbf{113} (1993),
  no.~3, 615--638.

\bibitem[Khi35]{khintchine1935}
A.~Khintchine, \emph{Metrische {K}ettenbruchprobleme}, Compositio Math.
  \textbf{1} (1935), 361--382.

\bibitem[KM92]{kesten1992ratios}
H.~Kesten and R.~A. Maller, \emph{Ratios of trimmed sums and order statistics},
  Ann. Probab. \textbf{20} (1992), no.~4, 1805--1842.

\bibitem[KM95]{km2}
\bysame, \emph{The effect of trimming on the strong law of large numbers},
  Proc. London Math. Soc. (3) \textbf{71} (1995), no.~2, 441--480.

\bibitem[KS19]{KESSEBOHMER20194163}
M.~Kesseb\"ohmer and T.~I. Schindler, \emph{Strong laws of large numbers for
  intermediately trimmed {B}irkhoff sums of observables with infinite mean},
  Stochastic Process. Appl. \textbf{129} (2019), no.~10, 4163--4207.

\bibitem[KS20]{kessschimean}
\bysame, \emph{Mean convergence for intermediately trimmed {B}irkhoff sums of
  observables with regularly varying tails}, Nonlinearity \textbf{33} (2020),
  no.~10, 5543--5566. 

\bibitem[Lam63]{LAMPERTIBC}
J.~Lamperti, \emph{Wiener's test and {M}arkov chains}, J. Math. Anal. Appl.
  \textbf{6} (1963), 58--66.

\bibitem[Mal84]{mallerstable}
R.~A. Maller, \emph{Relative stability of trimmed sums}, Z. Wahrsch. Verw.
  Gebiete \textbf{66} (1984), no.~1, 61--80.

\bibitem[Mor76]{morilighttrim}
T.~Mori, \emph{The strong law of large numbers when extreme terms are excluded
  from sums}, Z. Wahrscheinlichkeitstheorie und Verw. Gebiete \textbf{36}
  (1976), no.~3, 189--194.

\bibitem[Mor77]{mori2}
\bysame, \emph{Stability for sums of i.i.d. random variables when extreme terms
  are excluded}, Z. Wahrscheinlichkeitstheorie und Verw. Gebiete \textbf{40}
  (1977), no.~2, 159--167.

\bibitem[Pet02]{PETROVBC}
V.~V. Petrov, \emph{A note on the {B}orel-{C}antelli lemma}, Statist. Probab.
  Lett. \textbf{58} (2002), no.~3, 283--286.

\bibitem[Sch18]{Schindler2018TrimmedSF}
T.~I. Schindler, \emph{Trimmed sums for observables on the doubling map}, 2018, preprint. \href{https://arxiv.org/abs/1810.03223}{arXiv:1810.03223}

\bibitem[SU08]{sinai2008limittheorembirkoffsums}
Y.~G. Sinai and C.~Ulcigrai, \emph{A limit theorem for {B}irkhoff sums of
  non-integrable functions over rotations}, Geometric and probabilistic
  structures in dynamics, Contemp. Math., vol. 469, Amer. Math. Soc.,
  Providence, RI, 2008, pp.~317--340.

\end{thebibliography}
\end{document}